\documentclass[a4paper,12pt]{article}

\usepackage{latexsym}
\usepackage{amsmath}
\usepackage{amsthm}
\usepackage{amssymb}
\usepackage{enumerate}
\usepackage{multicol}
\usepackage{graphicx}

\allowdisplaybreaks

\theoremstyle{plain}
\newtheorem{theorem}{Theorem}[section]

\newtheorem{lemma}[theorem]{Lemma}
\newtheorem{corollary}[theorem]{Corollary}

\theoremstyle{definition}

\newtheorem{example}[theorem]{Example}
\newtheorem{remark}[theorem]{Remark}
\numberwithin{equation}{section}

\def\T{{\mathbb{T}}}
\def\R{{\mathbb{R}}}
\def\N{{\mathbb{N}}}

\def\Z{{\mathbb{Z}}}

\def\cD{{\mathcal{D}}}

\def\<{{\langle}}
\def\>{{\rangle}}

\def\liminf{\mathop{\mathrm{liminf}}}

\def\div{\mathop{\mathrm{div}}\nolimits}

\def\supp{\mathop{\mathrm{supp}}\nolimits}

\def\b0{{\mathbf{0}}}

\def\bef{{\mathbf{f}}}
\def\bg{{\mathbf{g}}}
\def\bG{{\mathbf{G}}}

\def\bn{{\mathbf{n}}}
\def\bx{{\mathbf{x}}}
\def\by{{\mathbf{y}}}
\def\bz{{\mathbf{z}}}
\def\bu{{\mathbf{u}}}
\def\bv{{\mathbf{v}}}
\def\bw{{\mathbf{w}}}

\def\bs{{\mathbf{s}}}

\def\bh{{\mathbf{h}}}

\def\wf{{\widetilde{f}}}

\def\bphi{{\mbox{\boldmath$\phi$}}}

\def\bpsi{{\mbox{\boldmath$\psi$}}}

\def\O{{\Omega}}
\def\Op{{\Omega_{\mathrm{per}}}}
\def\o{{\omega}}

\def\eps{{\varepsilon}}
\def\g{{\gamma}}
\def\s{{\sigma}}

\def\Dp{{\Delta_{\mathrm{per}}}}
\def\Dd{{\Delta_{\mathrm{D}}}}
\def\Qd{{Q_{\mathrm{D}}}}
\def\Rd{{R_{\mathrm{D}}}}
\begin{document}

\title{Characterization of Subdifferentials of a Singular Convex
Functional in Sobolev Spaces of Order Minus One}
\author{Yohei Kashima \medskip \\
Department of Mathematical Sciences, University of Tokyo\\
Komaba, Tokyo, 153-8914, Japan\\
kashima@ms.u-tokyo.ac.jp}
\date{}

\maketitle

\begin{abstract}
Subdifferentials of a singular convex functional 
 representing the surface free energy of a crystal under the roughening
 temperature are characterized. The energy functional is defined on
 Sobolev spaces of order $-1$, so the subdifferential mathematically
 formulates the energy's gradient which formally involves 4th order 
spacial derivatives of the surface's height. The subdifferentials are
 analyzed in the negative Sobolev spaces of arbitrary spacial dimension 
on which both a periodic
 boundary condition and a Dirichlet boundary condition are separately
 imposed. Based on the characterization theorem of subdifferentials, the smallest element
 contained in the subdifferential of the energy for a spherically symmetric surface is
 calculated under the Dirichlet boundary condition.
\end{abstract}

\section{Introduction}

When a time evolution problem has a structure of gradient flow and its
governing energy functional has good properties such as convexity and
lower semi-continuity, the evolution problem can be formulated into a
well-posed initial value problem whose right hand side is given by 
subdifferential of the energy functional. An advantage of the
subdifferential formulation is that smoothness of the energy functional
is not required, enabling us to handle a large class of physical models,
which only have a formal meaning at most, within mathematical
context. However, this mathematical formulation might look too abstract
to extract physical insights which the model is initially expected to
present. The abstract appearance is mainly due to the multi-valued
nature of subdifferential. In this formulation the time
derivative of unknown is not described by an equality, but is only contained
in a set of possible gradients of the energy at the time. This ambiguity
motivates us to characterize the subdifferential of the singular
functional explicitly so that one can interpret the abstract evolution
problem involving subdifferential as a natural formulation
of the original singular model. 

Our intention is especially to give an interpretation to the
subdifferential formulation of the following 4th order equation.
\begin{equation}\label{eq_singular_pde}
\frac{\partial}{\partial t}f=-\Delta \div(|\nabla f|^{-1}\nabla f+\mu
 |\nabla f|^{p-2}\nabla f),\ (\mu>0,p\in(1,\infty)),
\end{equation}
where $f$ is a time-dependent, real-valued function defined on a bounded
domain $\O$ of $\R^d$ obeying an appropriate boundary condition.
 Apparently the
equation \eqref{eq_singular_pde} loses a mathematical meaning when
$\nabla f=\b0$. However, if we put the mathematical rigor aside
temporarily, we can go on to rewrite the equation
\eqref{eq_singular_pde} symbolically into a gradient flow equation
\begin{equation}\label{eq_singular_pde_derivative}
\frac{\partial}{\partial t}f=-\frac{\delta F(f)}{\delta f}
\end{equation}
governed by the energy functional
\begin{equation}\label{eq_singular_energy}
F(f)=\int_{\O}\left(|\nabla f(\bx)|+\frac{\mu}{p}|\nabla
		f(\bx)|^p\right)d\bx.
\end{equation}
Here the functional derivative of $F$ is taken with respect to the
metric of the space $H^{-1}(\O)$ so that
$$\frac{\delta F(f)}{\delta f}=\Delta \div(|\nabla f|^{-1}\nabla f+\mu
 |\nabla f|^{p-2}\nabla f).
$$
 Recall that if we choose a Dirichlet boundary condition for instance, 
$H^{-1}(\O)$ is defined as the dual space of $H_0^{1}(\O)$. Using the
isometry $-\Delta:H_0^1(\O)\to H^{-1}(\O)$, we can formally regard 
$H^{-1}(\O)$ as a Hilbert space having the inner product
$\int_{\O}(-\Delta)^{-1}f(\bx)\cdot g(\bx)d\bx$ $(f,g\in H^{-1}(\O))$.
The function spaces will be defined later in this section in more
rigorous context.

The idea of the subdifferential formulation is simply to replace the
formal functional derivative by the subdifferential of $F$. The
formulation of \eqref{eq_singular_pde_derivative} is 
\begin{equation}\label{eq_singular_pde_formulation}
\frac{d}{d t}f\in -\partial F(f).
\end{equation}
We wish to postpone the mathematical definition of subdifferential
until the following subsections. Here let us only note that 
subdifferential is an extended concept of derivative since
its value is no other than the usual derivative if the functional is
differentiable. The strength of the abstract theory guarantees the
unique solvability of the initial value problem of
\eqref{eq_singular_pde_formulation}. In this paper we characterize the
value of $\partial F(f)$ so that we can regain a visible expression like
\eqref{eq_singular_pde} from \eqref{eq_singular_pde_formulation}.

Physically the solution $f$ to the equation \eqref{eq_singular_pde}
models the height of a crystalline surface driven by surface diffusion 
under the roughening temperature. Spohn \cite{S} systematically  derived the
equation \eqref{eq_singular_pde} and formulated it into a free boundary
value problem with evolving facets. Kashima \cite{K} proposed the
subdifferential formulation \eqref{eq_singular_pde_formulation} of the
singular problem \eqref{eq_singular_pde} under the Dirichlet boundary
condition and characterized the
subdifferential of the energy by revising the characterization theorem of
subdifferentials for 2nd order equations by Attouch and Damlamian
\cite{AD}. Odisharia \cite[\mbox{Chapter 3}]{O} derived a free boundary value problem,
which is consistent with Spohn's free boundary formulation \cite{S},
from the subdifferential formulation by Kashima \cite{K}. Odisharia's 
derivation excludes a speculation by Kashima in \cite{K} that the subdifferential formulation
of \eqref{eq_singular_pde} is inconsistent with the free boundary value
problem with facets. Developments on the subject have been
continuing until today. Recently Giga and Kohn \cite{GK} proved that the
solution to the initial value problem of
\eqref{eq_singular_pde_formulation} under the periodic boundary condition
becomes uniformly zero in finite time and obtained an upper bound on the
extinction time independently of the volume of the domain.
Kohn and Versieux \cite{KV} proposed a finite element approximation of
\eqref{eq_singular_pde_formulation} and established an error estimate
between the solution to \eqref{eq_singular_pde_formulation} and the
fully discrete finite element solution. More topics on singular diffusion equations including
\eqref{eq_singular_pde} are found in the article \cite{GG}.

This paper improves the previous results in \cite{K}. The article \cite{K} tried to characterize
$H^{-1}$-subdifferentials of a class of convex functionals including
\eqref{eq_singular_energy} under the Dirichlet boundary
condition in a way parallel to the general $L^2$-theory \cite{AD}.
In this paper by restricting the argument to the functional
\eqref{eq_singular_energy} we construct our proofs in a self-contained
manner using only a few basic facts from convex analysis 
 and characterize its $H^{-1}$-subdifferentials under both the
periodic boundary condition and the Dirichlet boundary condition
separately. The characterization is carried out in arbitrary spacial
dimension, improving the results in \cite{K}, where the dimension is
assumed to be less than equal to $4$. 
In addition to the removal of the dimensional constraint, the
characterized value of the subdifferential seems more natural especially
in the periodic setting as a formulation of \eqref{eq_singular_pde}.
The main task in our proof is to characterize the conjugate functional
of the energy functional and a technical difference from the argument
\cite[\mbox{Subsection 3.3}]{K} lies in this part, too. Though it was
also aimed to simplify the proof of the characterization of the
conjugate functional of \eqref{eq_singular_energy} in
\cite[\mbox{Subsection 3.3}]{K}, its argument needed the Sobolev
embedding theorem and consequently characterized the conjugate
functional under a restrictive assumption on the exponent $p$. In this
paper we complete the characterization of the conjugate functional
for all $p>1$. Remark that this approach is different from the method used
 to characterize $L^2$-subdifferential of total variation in
\cite[\mbox{Chapter 1}]{ACM}, which is based on a fact that the
functional of total variation is positive homogeneous of degree 1.
By applying the characterization theorem we calculate the smallest element
in the subdifferential of the energy functional under the Dirichlet
boundary condition for a spherically symmetric surface in any spacial dimension. The smallest
element is called canonical restriction. Our calculation of the
canonical restriction is seen as an extension of that of 1 dimensional case
 presented in \cite[\mbox{Section 4}]{K} for the Dirichlet problem,
 \cite[\mbox{Chapter 3}]{O} for the periodic problem. The canonical
 restriction is relevant to
 the study of the crystalline motion since the general theory
 (see e.g. \cite{B}) suggests that it actually represents the speed of
 the surface during the time evolution. From the canonical restriction
 we can, therefore, predict how the surface behaves in the next moment,
 which was in fact the strategy of Odisharia \cite[\mbox{Chapter 3}]{O}
 to derive the free boundary value problem.

In the rest of this section we prepare notations, introduce function
spaces, and state the main results concerning the characterization of subdifferentials. In Section \ref{sec_proofs} we give proofs
of the characterization theorems first for the periodic problem, then for the
Dirichlet problem. In Section \ref{sec_canonical_restriction} we
calculate the canonical restriction under the
Dirichlet boundary by assuming a spherical symmetry of the surface.

\subsection{Function spaces with a periodic boundary condition}
Here we introduce notations and function spaces to formulate the
periodic problem. Throughout the paper the number $d(\in\N)$ denotes the
spacial dimension and $p(\in(1,\infty))$ is used to define the exponent of the
spaces of
integrable functions. The notation $\T^d$ stands for a $d$-dimensional
flat torus; $\T^d:=\prod_{i=1}^d(\R/\o_i\Z)$ with $\o_i>0$
$(i=1,2,\cdots,d)$. Set $\Op:=\prod_{i=1}^d(0,\o_i)$ $(\subset \R^d)$.

We consider the following real Banach space of periodic integrable
functions.
\begin{equation*}
L^p(\T^d;\R^m):=\left\{\bef\in L_{\text{loc}}^p(\R^d;\R^m)\ \left|\begin{array}{l}
 \bef(\bx)=\bef(\bx+(m_1\o_1,\cdots,m_d\o_d))\\
 \text{a.e. }\bx\in\R^d,\ \forall (m_1,\cdots,m_d)\in\Z^d
\end{array}\right.
\right\},
\end{equation*}
where $m\in\N$ and 
the notation $\bef\in L_{\text{loc}}^p(\R^d;\R^m)$ means that for
any open bounded set $O$ $(\subset \R^d)$, $\bef|_{O}\in
L^p(O;\R^m)$. The norm of $L^p(\T^d;\R^m)$ is defined by 
$$\|\bef\|_{L^p(\T^d;\R^m)}:=\left(\int_{\Op}|\bef(\bx)|^pd\bx\right)^{1/p}.
$$
Among these spaces $L^2(\T^d;\R^m)$ is a Hilbert space having the inner product
$$\<\bef,\bg\>_{L^2(\T^d;\R^m)}:=\int_{\Op}\<\bef(\bx),\bg(\bx)\>_{\R^m}d\bx.
$$
When $m=1$, let us simply write $L^p(\T^d)$ instead of $L^p(\T^d;\R)$.

The space $L^p_{\text{ave}}(\T^d)$ is a subspace of $L^p(\T^d)$ defined
by
$$L_{\text{ave}}^p(\T^d):=\left\{f\in L^p(\T^d)\ \left|\
\int_{\Op}f(\bx)d\bx =0\right.\right\}.$$

The Sobolev spaces $W^{1,p}(\T^d)$, $W^{1,p}_{\text{ave}}(\T^d)$ are defined
by
\begin{equation*}
\begin{split}
&W^{1,p}(\T^d):=\left\{f\in L^p(\T^d)\ \left|\
					  \nabla f(\in\cD'(\R^d;\R^d))\text{
					  satisfies }\nabla f\in
 L^p(\T^d;\R^d)\right.\right\},\\
&W^{1,p}_{\text{ave}}(\T^d):= W^{1,p}(\T^d)\cap
 L^p_{\text{ave}}(\T^d).
\end{split}
\end{equation*}
We use the notation $H^1_{\text{ave}}(\T^d)$ in place of
$W^{1,2}_{\text{ave}}(\T^d)$.

Poincar\'e's inequality states that there exists a constant $C$ $(>0)$
such that 
$$\|f\|_{L^p(\T^d)}\le C\|\nabla f\|_{L^p(\T^d;\R^d)},\ \forall
f\in W^{1,p}_{\text{ave}}(\T^d).$$ 
This inequality enables
us to adapt $\|\nabla \cdot\|_{L^p(\T^d;\R^d)}$ as the norm of
$W^{1,p}_{\text{ave}}(\T^d)$ and $\<\nabla\cdot,\nabla\cdot\>_{L^2(\T^d;\R^d)}$
 as the inner product of the Hilbert space $H^1_{\text{ave}}(\T^d)$.

Throughout the paper we use the notation $\<\cdot,\cdot\>$ to denote the
scalar product of duality between a real Banach space and its
topological dual space. We do not specify which duality is being
described by $\<\cdot,\cdot\>$ if it is clear from the context.

Let $H^{-1}_{\text{ave}}(\T^d)$ denote the topological dual space of
$H^1_{\text{ave}}(\T^d)$. We define a linear operator
$-\Dp:H^1_{\text{ave}}(\T^d)\to H^{-1}_{\text{ave}}(\T^d)$ by
$$\<-\Dp f,\cdot\>:=\<\nabla f,\nabla \cdot\>_{L^2(\T^d;\R^d)},\
\forall f\in H^1_{\text{ave}}(\T^d).$$
Because of our choice of the inner product of $H^1_{\text{ave}}(\T^d)$
and Riesz' representation theorem, the operator
$-\Dp:H^1_{\text{ave}}(\T^d)\to H^{-1}_{\text{ave}}(\T^d)$ is an
isometry. The dual space $H^{-1}_{\text{ave}}(\T^d)$ can be considered
as a Hilbert space equipped with the inner product
$\<\cdot,\cdot\>_{H^{-1}_{\text{ave}}(\T^d)}$ defined by 
$$\<f,g\>_{H^{-1}_{\text{ave}}(\T^d)}:=\<(-\Dp)^{-1}f,g\>,\
 \forall f,g\in  H^{-1}_{\text{ave}}(\T^d). 
$$

Introduce the space of smooth periodic functions by
\begin{equation*}
C^{\infty}(\T^d;\R^m):=\left\{\bef\in C^{\infty}(\R^d;\R^m)\ \left|\begin{array}{l}
 \bef(\bx)=\bef(\bx+(m_1\o_1,\cdots,m_d\o_d)),\\
 \forall\bx\in\R^d,\ \forall (m_1,\cdots,m_d)\in\Z^d
\end{array}\right.
\right\}.
\end{equation*}
Again let us simply write $C^{\infty}(\T^d)$ instead of
$C^{\infty}(\T^d;\R)$.
We define a subspace of $C^{\infty}(\T^d)$ by
$$C^{\infty}_{\text{ave}}(\T^d):=\left\{f\in C^{\infty}(\T^d)\ \left|\
\int_{\Op}f(\bx)d\bx =0\right.\right\}.
$$
We will make use of the following density property.
\begin{lemma}\label{lem_density}
The set $C^{\infty}_{\mathrm{ave}}(\T^d)$ is dense in
 $W^{1,p}_{\mathrm{ave}}(\T^d)$.
\end{lemma}
\begin{proof}
Let $\rho\in C_0^{\infty}(\R^d)$ be such that
\begin{equation}\label{eq_mollifier_kernel}
\rho(\bx)\ge 0\ (\forall\bx\in\R^d),\ 
\rho(\bx)=0\text{ if }|\bx|\ge 1,\ 
\int_{\R^d}\rho(\bx)d\bx=1.
\end{equation}
For any $f\in W^{1,p}_{\text{ave}}(\T^d)$ and $\delta>0$ define a
 function $f_{\delta}:\R^d\to\R$ by
$$f_{\delta}(\bx):=\int_{\R^d}\delta^{-d}\rho\left(\frac{\bx-\by}{\delta}\right)f(\by)d\by.
$$
By using standard properties of the mollifier and the periodicity of $f$
 one can check that $f_{\delta}\in C^{\infty}_{\text{ave}}(\T^d)$ and
 $f_{\delta}$ converges to $f$ in $W^{1,p}_{\text{ave}}(\T^d)$ as
 $\delta\searrow 0$.
\end{proof}

Remark that these spaces of periodic functions are equivalent to those axiomatically
defined on the compact Riemannian manifold $\T^d$, the flat torus. See
e.g. \cite{J} for the construction of $\T^d$ as a Riemannian manifold and
\cite{Au} for Sobolev spaces on Riemannian manifolds in general.

We define a subset $X_{\text{per}}$ of $H^{-1}_{\text{ave}}(\T^d)$ as
follows. An $f$ $(\in H^{-1}_{\text{ave}}(\T^d))$ belongs to
$X_{\text{per}}$ if there exists $\wf\in
W^{1,p}_{\text{ave}}(\T^d)$ such that 
$$
\<f,\phi\>=\lim_{n\to\infty}\int_{\Op}\wf(\bx)\phi_n(\bx)d\bx,\
 \forall\phi\in H^{1}_{\text{ave}}(\T^d),
$$
where $\{\phi_n\}_{n=1}^{\infty}$ $(\subset
C^{\infty}_{\text{ave}}(\T^d))$ is any sequence converging to $\phi$ in
$H^{1}_{\text{ave}}(\T^d)$ as $n\to \infty$.

Note that for any $f\in X_{\mathrm{per}}$ such $\wf$ $(\in
W^{1,p}_{\text{ave}}(\T^d))$ uniquely exists. From now we use the
notation `` $\widetilde{\cdot}$ '' to indicate the corresponding function
of $W^{1,p}_{\mathrm{ave}}(\T^d)$ to a given element of
$X_{\mathrm{per}}$. It follows that $X_{\mathrm{per}}$ is a real linear
space and the map $f\mapsto
\wf:X_{\mathrm{per}}\to W^{1,p}_{\mathrm{ave}}(\T^d)$ is linear.

By using these notions we now define the functional
$F_{\mathrm{per}}:H^{-1}_{\mathrm{ave}}(\T^d)\to \R\cup \{\infty\}$ by
$$
F_{\mathrm{per}}(f):=\left\{\begin{array}{ll}\displaystyle\int_{\Op}\s(\nabla\wf(\bx))d\bx
		      &\text{if }f\in X_{\mathrm{per}},\\
\infty & \text{otherwise,}
\end{array}
\right.
$$
where $\s:\R^d\to\R$ is defined by
$$\s(\by):=|\by|+\frac{\mu}{p} |\by|^p, \ (\mu >0,\ p\in (1,\infty)).$$

\begin{lemma}\label{lem_F_per_lsc}
The functional $F_{\mathrm{per}}:H^{-1}_{\mathrm{ave}}(\T^d)\to \R\cup
 \{\infty\}$ is convex, lower semi-continuous and not identically
 $\infty$.
\end{lemma}
\begin{proof}
Being convex and not identically $\infty$ can be seen from the
 definition. To show the lower semi-continuity of $F_{\mathrm{per}}$,
 assume that $\{f_n\}_{n=1}^{\infty}$ $(\subset
 H^{-1}_{\mathrm{ave}}(\T^d))$ converges to $f$ in
 $H^{-1}_{\mathrm{ave}}(\T^d)$ as $n\to \infty$ and
 $F_{\mathrm{per}}(f_n)\le \lambda$ $(\forall n\in \N)$, where
 $\lambda\ge 0$. 

Since $\{\wf_n\}_{n=1}^{\infty}$ is bounded in
 $W^{1,p}_{\mathrm{ave}}(\T^d)$, there are $g\in
 W^{1,p}_{\mathrm{ave}}(\T^d)$ and a subsequence
 $\{\wf_{n(j)}\}_{j=1}^{\infty}$ of $\{\wf_{n}\}_{n=1}^{\infty}$ such that $\wf_{n(j)}$ weakly converges to
 $g$ in $W^{1,p}_{\mathrm{ave}}(\T^d)$ as $j\to \infty$. Mazur's theorem 
for $H^{-1}_{\mathrm{ave}}(\T^d)\times W^{1,p}_{\mathrm{ave}}(\T^d)$
 guarantees that for any $k\in \N$ there exist
 $j_k\in\N$ and $\alpha_l^k\in [0,1]$ $(l=1,\cdots,j_k)$ satisfying
 $\sum_{l=1}^{j_k}\alpha_l^k=1$ such that as $k\to \infty$
$$
\sum_{l=1}^{j_k}\alpha_l^kf_{n(l)}\to f\text{ in
 }H^{-1}_{\mathrm{ave}}(\T^d),\ \sum_{l=1}^{j_k}\alpha_l^k\wf_{n(l)}\to g\text{ in }W^{1,p}_{\mathrm{ave}}(\T^d).
$$
Moreover, for any $\psi \in C_{\mathrm{ave}}^{\infty}(\T^d)$
\begin{equation*}
\<f,\psi\>=\lim_{k\to\infty}\<\sum_{l=1}^{j_k}\alpha_l^k
 f_{n(l)},\psi\>=\lim_{k\to\infty}\int_{\Op}\sum_{l=1}^{j_k}\alpha_l^k
 \wf_{n(l)}(\bx)\psi(\bx)d\bx=\int_{\Op}g(\bx)\psi(\bx)d\bx.
\end{equation*}
Hence, for any $\phi\in H^{1}_{\mathrm{ave}}(\T^d)$ and
 $\{\phi_n\}_{n=1}^{\infty}$ $(\subset  C_{\mathrm{ave}}^{\infty}(\T^d))$
 converging to $\phi$ in $H^{1}_{\mathrm{ave}}(\T^d)$
$$
\<f,\phi\>=\lim_{n\to \infty}\<f,\phi_n\>=\lim_{n\to \infty}\int_{\Op}g(\bx)\phi_n(\bx)d\bx,
$$
which means that $f\in X_{\mathrm{per}}$ and $g=\wf$.

Then by the convexity and the continuity of $\int_{\Op}\s(\cdot)d\bx$
 in $L^p(\T^d;\R^d)$
\begin{equation*}
\begin{split}
F_{\mathrm{per}}(f)&=\int_{\Op}\s(\nabla \wf(\bx))d\bx=\lim_{k\to
 \infty}\int_{\Op}\s\left(\sum_{l=1}^{j_k}\alpha_l^k\nabla
 \wf_{n(l)}(\bx)\right)d\bx\\
&\le \limsup_{k\to
 \infty}\sum_{l=1}^{j_k}\alpha_l^kF_{\mathrm{per}}(f_{n(l)})\le \lambda,
\end{split}
\end{equation*}
which concludes that $F_{\mathrm{per}}$ is lower semi-continuous in $H^{-1}_{\mathrm{ave}}(\T^d)$.
\end{proof}

\subsection{Function spaces with a Dirichlet boundary condition}
Here we prepare some notions necessary to formulate the Dirichlet
problem. Let $\O$ be an open bounded subset of $\R^d$. By Poincar\'e's
inequality we may choose $\|\nabla \cdot\|_{L^p(\O;\R^d)}$ as the norm
of $W_0^{1,p}(\O)$ and $\<\nabla \cdot,\nabla \cdot\>_{L^2(\O;\R^d)}$
 as the inner product of $H^1_0(\O)$. Let $H^{-1}(\O)$ denote the
 topological dual space of the
Hilbert space $H^1_0(\O)$. We define a linear map $-\Dd:H^1_0(\O)\to
H^{-1}(\O)$ by
$$
\<-\Dd f,\cdot\>:=\<\nabla f,\nabla \cdot\>_{L^2(\O;\R^d)},\  \forall f\in
H_0^1(\O).
$$
By using Riesz' representation theorem we can prove that the linear map
$-\Dd:H^1_0(\O)\to H^{-1}(\O)$ is an isometry. The dual space
$H^{-1}(\O)$ is a Hilbert space having the inner product $\<\cdot,\cdot\>_{H^{-1}(\O)}$ defined by
$$
\<f,g\>_{H^{-1}(\O)}:=\<(-\Dd)^{-1}f,g\>,\ \forall f,g\in H^{-1}(\O).
$$

Let $X_{\mathrm{D}}$ denote a subset of $H^{-1}(\O)$ consisting of any $f\in
H^{-1}(\O)$ for which there exists $\wf\in W_0^{1,p}(\O)$ such that 
$$
\<f,\phi\>=\lim_{n\to \infty}\int_{\O}\wf(\bx)\phi_n(\bx)d\bx,\ \forall
\phi\in H^1_0(\O),
$$
where $\{\phi_n\}_{n=1}^{\infty}$ $(\subset C_0^{\infty}(\O))$ is any
sequence converging to $\phi$ in $H^1_0(\O)$ as $n\to\infty$. For given
$f\in X_{\mathrm{D}}$ such $\wf$ $(\in W_0^{1,p}(\O))$ uniquely
exists. As in the periodic case we use the notation ``
$\widetilde{\cdot}$ '' 
to represent the function of $W_0^{1,p}(\O)$
associated with a given element of $X_{\mathrm{D}}$.

We define the functional $F_{\mathrm{D}}:H^{-1}(\O)\to \R\cup\{\infty\}$
by
$$
F_{\mathrm{D}}(f):=\left\{\begin{array}{ll}\displaystyle\int_{\O}\s(\nabla\wf(\bx))d\bx&\text{if
		    }f\in X_{\mathrm{D}},\\
\infty&\text{otherwise.}
\end{array}
\right.
$$
The following lemma can be proved in the same way as in Lemma
\ref{lem_F_per_lsc}.
\begin{lemma}\label{lem_F_D_lsc}
The functional $F_{\mathrm{D}}:H^{-1}(\O)\to \R\cup\{\infty\}$ is
 convex, lower semi-continuous and not identically $\infty$.
\end{lemma}

\subsection{Subdifferentials}
Subdifferential is an extended concept of differential. Subdifferential
of a functional becomes a multi-valued operator if the functional is not
differentiable in the normal sense. Let us see this by calculating
the subdifferential of the energy density $\s$. The subdifferential
$\partial \s(\cdot):\R^d\to 2^{\R^d}$ is defined by
\begin{equation*}
\partial \s(\bx):=\{\by\in\R^d\ |\
 \<\by,\bz\>_{\R^d}+\s(\bx)\le \s(\bx+\bz),\ \forall \bz\in\R^d\},\
 \forall \bx\in\R^d.
\end{equation*}
It follows directly from the definition that
$$\partial
\s(\bx)=\left\{\begin{array}{ll}\{|\bx|^{-1}\bx +\mu|\bx|^{p-2}\bx\}& \text{
	 if }\bx\neq\b0,\\
\{\by\in\R^d\ |\ |\by|\le 1\}& \text{ if }\bx=\b0.
\end{array}
\right.$$
From this characterization we see that if $\bx\neq \b0$ the only element
of $\partial \s(\bx)$ is nothing but the gradient of $\s(\cdot)$ at
$\bx$. However, at $\bx=\b0$, where $\s(\cdot)$ is not differentiable,
$\partial\s(\bx)$ becomes multi-valued.

We define the subdifferential $\partial
F_{\mathrm{per}}(\cdot):H^{-1}_{\mathrm{ave}}(\T^d)\to
2^{H^{-1}_{\mathrm{ave}}(\T^d)}$ of $F_{\mathrm{per}}$ by
$$
\partial F_{\mathrm{per}}(f):=\{g\in H^{-1}_{\mathrm{ave}}(\T^d)\ |\
\<g,h\>_{H^{-1}_{\mathrm{ave}}(\T^d)}+F_{\mathrm{per}}(f)\le
F_{\mathrm{per}}(f+h),\ \forall h\in H^{-1}_{\mathrm{ave}}(\T^d)\}
$$
and the subdifferential $\partial F_{\mathrm{D}}(\cdot):H^{-1}(\O)\to
2^{H^{-1}(\O)}$ of $F_{\mathrm{D}}$ by
$$
\partial F_{\mathrm{D}}(f):=\{g\in H^{-1}(\O)\ |\
\<g,h\>_{H^{-1}(\O)}+F_{\mathrm{D}}(f)\le
F_{\mathrm{D}}(f+h),\ \forall h\in H^{-1}(\O)\}.
$$
Our main purpose is to characterize $\partial F_{\mathrm{per}}(\cdot)$ and $\partial
F_{\mathrm{D}}(\cdot)$. The results are the following.

\begin{theorem}\label{thm_subdifferential_per}
If $\partial F_{\mathrm{per}}(f)\neq \emptyset$,
\begin{equation*}
\partial F_{\mathrm{per}}(f)=\left\{-(-\Dp)\div \bg\
 \left|\begin{array}{l}\bg\in L^{p/(p-1)}(\T^d;\R^d) \text{ satisfying
  that }\\ \div\bg\in
  H_{\mathrm{ave}}^1(\T^d),\\ \bg(\bx)\in\partial\s(\nabla \wf(\bx))
   \text{ a.e. }\bx\in\R^d\end{array}
\right.\right\}.
\end{equation*}
\end{theorem}

\begin{theorem}\label{thm_subdifferential_D}
If $\partial F_{\mathrm{D}}(f)\neq \emptyset$,
\begin{equation*}
\partial F_{\mathrm{D}}(f)=\left\{-(-\Dd)\div \bg\
 \left|\begin{array}{l}\bg\in L^{p/(p-1)}(\O;\R^d) \text{ satisfying
  that }\\ \div\bg\in
  H_{0}^1(\O),\\ \bg(\bx)\in\partial\s(\nabla \wf(\bx))
   \text{ a.e. }\bx\in\O,\\
\<h,\div\bg\>+\int_{\O}\<\nabla
 \widetilde{h}(\bx),\bg(\bx)\>_{\R^d}d\bx=0,\ \forall h\in X_{\mathrm{D}}
\end{array}
\right.\right\}.
\end{equation*}
\end{theorem}

By assuming an additional condition on $p(\in(1,\infty))$ we can simplify the characterization of
 Theorem \ref{thm_subdifferential_D} as follows.

\begin{corollary}\label{cor_subdifferential_D_restricted}
Assume that 
\begin{equation}\label{eq_condition_D}
\begin{array}{ll}
p\in(1,\infty)&\text{if }d\le 4,\\
p\in\left[\displaystyle\frac{2d}{d+4}, \infty\right)&\text{if }d\ge 5.
\end{array}
\end{equation}
If $\partial F_{\mathrm{D}}(f)\neq \emptyset$,
\begin{equation*}
\partial F_{\mathrm{D}}(f)=\left\{-(-\Dd)\div \bg\
 \left|\begin{array}{l}\bg\in L^{p/(p-1)}(\O;\R^d) \text{ satisfying
  that }\\ \div\bg\in
  H_{0}^1(\O),\\ \bg(\bx)\in\partial\s(\nabla \wf(\bx))
   \text{ a.e. }\bx\in\O
\end{array}
\right.\right\}.
\end{equation*}
\end{corollary}
\begin{remark}
In Lemma \ref{lem_F_per_lsc} and Lemma \ref{lem_F_D_lsc} we have seen
 that both
 $F_{\mathrm{per}}:H_{\mathrm{ave}}^{-1}(\T^d)\to\R\cup\{\infty\}$ and
 $F_{\mathrm{D}}:H^{-1}(\O)\to\R\cup\{\infty\}$ are convex, lower
 semi-continuous, and not identically $\infty$. These properties are
 sufficient to ensure the unique solvability of the initial value
 problems to find $f_{\mathrm{per}}\in
 C([0,\infty);H_{\mathrm{ave}}^{-1}(\T^d))$ and $f_{\mathrm{D}}\in
 C([0,\infty);H^{-1}(\O))$ such that 
$$
\left\{\begin{array}{l}
\frac{d}{dt}f_{\mathrm{per}}(t)\in -\partial
 F_{\mathrm{per}}(f_{\mathrm{per}}(t))\text{ a.e.}t>0,\\
f_{\mathrm{per}}(0)=f_{\mathrm{per},0}(\in \overline{X_{\mathrm{per}}}),
\end{array}
\right.\ 
\left\{\begin{array}{l}
\frac{d}{dt}f_{\mathrm{D}}(t)\in -\partial
 F_{\mathrm{D}}(f_{\mathrm{D}}(t))\text{ a.e.}t>0,\\
f_{\mathrm{D}}(0)=f_{\mathrm{D},0}(\in \overline{X_{\mathrm{D}}})
\end{array}
\right.
$$
(see e.g. \cite{B}). Theorems above characterize the right hand sides
 of these evolution systems and provide us with explicit
 representations comparable to the right hand side of the original model 
 \eqref{eq_singular_pde}.
\end{remark}

\section{Proof of the characterization of subdifferentials}\label{sec_proofs}
In this section we prove Theorem \ref{thm_subdifferential_per}, Theorem
\ref{thm_subdifferential_D} and Corollary
\ref{cor_subdifferential_D_restricted}. Let us fix some notational
conventions and recall a few basic facts from convex analysis
beforehand. For a real Banach space $B$ let $B^*$ denote its topological
dual space. For a functional $E:B\to \R\cup\{\infty\}$ being not
identically $\infty$ its conjugate functional $E^*:B^*\to\R\cup\{\infty\}$ is
defined by
$$
E^*(v):=\sup_{u\in B}\{\<v,u\>-E(u)\},\ \forall v\in B^*.
$$ 

\begin{lemma}\label{lem_standard_properties_Banach}
Assume that $E:B\to \R\cup\{\infty\}$ is convex, lower semi-continuous
 and not identically $\infty$. The following hold true.
\begin{enumerate}[(1)]
\item\label{properties_conjugate} $E^*:B^*\to\R\cup \{\infty\}$ is convex, lower semi-continuous and
      not identically $\infty$.
\item\label{double_conjugate} $(E^*)^*(v)=E(v)$, $\forall v\in B$.
\end{enumerate}
\end{lemma}

For a functional defined on a real Hilbert space $H$ we adapt the inner
product $\<\cdot,\cdot\>_H$ to define its conjugate functional. To
distinguish from Banach spaces' case, let us change a notation.
For a functional $F:H\to \R\cup \{\infty\}$ being not identically
$\infty$ we define its conjugate functional $F^{\#}:H\to \R\cup\{\infty\}$
by
$$
F^{\#}(v):=\sup_{u\in H}\{\<v,u\>_H-F(u)\},\ \forall v\in H.
$$
Moreover, we define its subdifferential $\partial F:H\to 2^H$ by
$$
\partial F(u):=\{v\in H\ |\ \<v,w\>_H+F(u)\le F(u+w),\ \forall w\in H \}.
$$ 

\begin{lemma}\label{lem_standard_properties_Hilbert}
Assume that a functional $F:H\to \R\cup\{\infty\}$ is convex, lower
 semi-continuous and not identically $\infty$. The following statements
 are equivalent to each other.
\begin{enumerate}[(i)]
\item $v\in \partial F(u)$.
\item $u\in \partial F^{\#}(v)$.
\item $F(u)+F^{\#}(v)=\<u,v\>_H$.
\end{enumerate}
\end{lemma}
We use Lemma \ref{lem_standard_properties_Banach} and Lemma
\ref{lem_standard_properties_Hilbert} without providing the proofs. See
e.g. \cite{ET} to verify them.

The conjugate functional $\s^{\#}:\R^d\to \R$ of $\s$ and its
subdifferential $\partial \s^{\#}(\cdot):\R^d\to 2^{\R^d}$ can be
calculated from the definitions.

\begin{lemma}\label{lem_sigma_characterization}
For any $\by\in\R^d$
\begin{align}
&\s^{\#}(\by)=\left\{\begin{array}{ll} 0 & \text{if }|\by|\le 1,\\
                                       \displaystyle\left(1-\frac{1}{p}\right)\mu^{-1/(p-1)}(|\by|-1)^{p/(p-1)}
					& \text{if }|\by|> 1,
\end{array}\right.\label{eq_sigma_conjugate}\\
&\partial \s^{\#}(\by)=\left\{\begin{array}{ll}\{\b0\} &\text{if
			}|\by|\le 1,\\
\left\{\mu^{-1/(p-1)}(|\by|-1)^{1/(p-1)}|\by|^{-1}\by\right\} &\text{if }|\by|>1.
\end{array}
\right.
\label{eq_subdif_sigma_conjugate}
\end{align}
\end{lemma}

\subsection{Proof for the periodic problem}
We are going to characterize the subdifferential of the periodic energy
$F_{\mathrm{per}}$. We introduce the real Banach space
$H_{\mathrm{ave}}^{-1}(\T^d)\times L^p(\T^d;\R^d)$ having the norm
defined by
$\|(f,\bg)\|:=\|f\|_{H_{\mathrm{ave}}^{-1}(\T^d)}+\|\bg\|_{L^p(\T^d;\R^d)}$.
Define functionals $Q,R:H_{\mathrm{ave}}^{-1}(\T^d)\times
L^p(\T^d;\R^d)$ $\to\R\cup \{\infty\}$ by
\begin{equation*}
\begin{split}
&Q((f,\bg)):=\int_{\Op}\s(\bg(\bx))d\bx,\\
&R((f,\bg)):=\left\{\begin{array}{ll}0 & \text{if }f\in
	      X_{\mathrm{per}}\text{ and }\bg=\nabla \wf,\\ 
\infty &\text{otherwise.}
\end{array}
\right.
\end{split}
\end{equation*}
One can check that $Q,R$ are convex, lower semi-continuous, and not
identically $\infty$.

We define a linear map $\Phi_{p/(p-1)}:H_{\mathrm{ave}}^{-1}(\T^d)\times
L^{p/(p-1)}(\T^d;\R^d)\to (H_{\mathrm{ave}}^{-1}(\T^d)\times L^p(\T^d;\R^d))^*$
by
\begin{equation*}
\begin{split}
&\<\Phi_{p/(p-1)}((u,\bv)),(f,\bg)\>:=\<u,f\>_{H_{\mathrm{ave}}^{-1}(\T^d)}+\int_{\Op}\<\bv(\bx),\bg(\bx)\>_{\R^d}d\bx,\\
&\quad \forall (u,\bv)\in H_{\mathrm{ave}}^{-1}(\T^d)\times
L^{p/(p-1)}(\T^d;\R^d),\ \forall (f,\bg) \in H_{\mathrm{ave}}^{-1}(\T^d)\times
L^{p}(\T^d;\R^d).
\end{split}
\end{equation*}
The map $\Phi_{p/(p-1)}$ is an isomorphism between these Banach spaces. 

In our proof characterizing the conjugate
functional $F^{\#}_{\mathrm{per}}$ $(:H_{\mathrm{ave}}^{-1}(\T^d)\to
\R\cup\{\infty\})$ is crucial to characterize $\partial
F_{\mathrm{per}}$. The first step is the following.

\begin{lemma}\label{lem_F_per_conjugate_pre}
For any $u\in H_{\mathrm{ave}}^{-1}(\T^d)$
\begin{equation}\label{eq_F_per_conjugate_pre}
F_{\mathrm{per}}^{\#}(u)=(Q+R)^*(\Phi_{p/(p-1)}((u,\b0))).
\end{equation}
\end{lemma}
\begin{proof}
Take any $(u,\bv)\in H_{\mathrm{ave}}^{-1}(\T^d)\times
L^{p/(p-1)}(\T^d;\R^d)$.
\begin{equation*}
\begin{split}
(Q&+R)^*(\Phi_{p/(p-1)}((u,\bv)))\\
&=\sup_{(f,\bg)\in H_{\mathrm{ave}}^{-1}(\T^d)\times
L^{p}(\T^d;\R^d)}\{\<\Phi_{p/(p-1)}((u,\bv)),(f,\bg)\>-(Q+R)((f,\bg))\}\\
&=\sup_{f\in
 X_{\mathrm{per}}}\left\{\<u,f\>_{H_{\mathrm{ave}}^{-1}(\T^d)}
 +\int_{\Op}\<\bv(\bx),\nabla \wf(\bx)\>_{\R^d}d\bx-\int_{\Op}\s(\nabla
 \wf(\bx))d\bx\right\},
\end{split}
\end{equation*}
from which the claimed equality follows.
\end{proof}

We will characterize the right hand side of
\eqref{eq_F_per_conjugate_pre} after characterizing $Q^*$ and $R^*$. 

\begin{lemma}\label{lem_Q_conjugate}
For any $(u,\bv)\in H_{\mathrm{ave}}^{-1}(\T^d)\times
L^{p/(p-1)}(\T^d;\R^d)$
\begin{equation*}
Q^*(\Phi_{p/(p-1)}((u,\bv)))=\left\{\begin{array}{ll}\displaystyle\int_{\Op}\s^{\#}(\bv(\bx))d\bx
			       &\text{if }u=0,\\
\infty&\text{otherwise.}
\end{array}
\right.
\end{equation*}
\end{lemma}
\begin{proof}
Take any $(u,\bv)\in H_{\mathrm{ave}}^{-1}(\T^d)\times
L^{p/(p-1)}(\T^d;\R^d)$.
\begin{equation}\label{eq_Q_conjugate_obvious}
\begin{split}
&Q^*(\Phi_{p/(p-1)}((u,\bv)))\\
&=\sup_{(f,\bg)\in H_{\mathrm{ave}}^{-1}(\T^d)\times
L^{p}(\T^d;\R^d)}\{\<\Phi_{p/(p-1)}((u,\bv)),(f,\bg)\>-Q((f,\bg))\}\\
&=\sup_{(f,\bg)\in H_{\mathrm{ave}}^{-1}(\T^d)\times
 L^{p}(\T^d;\R^d)}\left\{\<u,f\>_{H_{\mathrm{ave}}^{-1}(\T^d)}+\int_{\Op}\<\bv(\bx),\bg(\bx)\>_{\R^d}d\bx-\int_{\Op}\s(\bg(\bx))d\bx\right\}\\
&=\left\{\begin{array}{ll}\displaystyle\sup_{\bg\in
   L^{p}(\T^d;\R^d)}\int_{\Op}(\<\bv(\bx),\bg(\bx)\>_{\R^d}-\s(\bg(\bx)))d\bx &\text{if }u=0,\\
\infty &\text{otherwise. }
\end{array}
\right.
\end{split}
\end{equation}

On one hand, it follows from the definition of $\s^{\#}$ that
\begin{equation}\label{eq_Q_conjugate_one}
\sup_{\bg\in
 L^{p}(\T^d;\R^d)}\int_{\Op}(\<\bv(\bx),\bg(\bx)\>_{\R^d}-\s(\bg(\bx)))d\bx\le
 \int_{\Op}\s^{\#}(\bv(\bx))d\bx.
\end{equation}

On the other hand, let us define $\bh\in L^{p}(\T^d;\R^d)$ by
\begin{equation*}
\bh(\bx):=\left\{\begin{array}{ll}\b0&\text{if }|\bv(\bx)|\le 1,\\
\mu^{-1/(p-1)}(|\bv(\bx)|-1)^{1/(p-1)}|\bv(\bx)|^{-1}\bv(\bx)&\text{if
 }|\bv(\bx)|> 1.
\end{array}
\right.
\end{equation*}
By \eqref{eq_subdif_sigma_conjugate}
\begin{equation}\label{eq_Q_conjugate_inclusion}
\bh(\bx)\in\partial \s^{\#}(\bv(\bx))\text{ a.e. }\bx\in\R^d.
\end{equation}
By Lemma \ref{lem_standard_properties_Hilbert} the inclusion
 \eqref{eq_Q_conjugate_inclusion} implies that
$$
\s^{\#}(\bv(\bx))=\<\bv(\bx), \bh(\bx)\>_{\R^d}-\s(\bh(\bx)) \text{
 a.e. }\bx\in\R^d,
$$
which leads to 
\begin{equation}\label{eq_Q_conjugate_theother}
\int_{\Op}\s^{\#}(\bv(\bx))d\bx\le \sup_{\bg\in
 L^p(\T^d;\R^d)}\int_{\Op}(\<\bv(\bx),\bg(\bx)\>_{\R^d}-\s(\bg(\bx)))d\bx.
\end{equation}
 By putting \eqref{eq_Q_conjugate_obvious}, \eqref{eq_Q_conjugate_one}
 and \eqref{eq_Q_conjugate_theother} together, we obtain the result.
\end{proof}

To characterize $R^*$ we need a couple of lemmas based on density properties
of smooth functions in the periodic Sobolev spaces.
\begin{lemma}\label{lem_integration_by_parts}
For any $ f\in W_{\mathrm{ave}}^{1,p}(\T^d)$ and $\bphi\in
 C^{\infty}(\T^d;\R^d)$
\begin{equation*}
\int_{\Op}f(\bx)\div\bphi(\bx)d\bx+\int_{\Op}\<\nabla
 f(\bx),\bphi(\bx)\>_{\R^d}d\bx=0.
\end{equation*}
\end{lemma}
\begin{proof}
Lemma \ref{lem_density} justifies the equality.
\end{proof}

\begin{lemma}\label{lem_density_div}
For any $\bv\in L^{p/(p-1)}(\T^d;\R^d)$ satisfying $\div \bv\in
 H_{\mathrm{ave}}^{1}(\T^d)$ there exists
 $\{\bv_n\}_{n=1}^{\infty}\subset C^{\infty}(\T^d;\R^d)$ such that
 as $n\to\infty$
\begin{equation*}
\begin{split}
&\bv_n\to \bv\text{ in }L^{p/(p-1)}(\T^d;\R^d),\\
&\div \bv_n\to \div \bv\text{ in }H_{\mathrm{ave}}^{1}(\T^d).
\end{split}
\end{equation*}
\end{lemma}
\begin{proof}
As in Lemma \ref{lem_density} let us define a function
 $\bv_{\delta}:\R^d\to \R^d$ by
$$
\bv_{\delta}(\bx):=\int_{\R^d}\delta^{-d}\rho\left(\frac{\bx-\by}{\delta}\right)\bv(\by)d\by
$$
 by choosing a function $\rho\in
 C_0^{\infty}(\R^d)$ having the properties \eqref{eq_mollifier_kernel}
 and $\delta>0$. The function $\bv_{\delta}$ is contained in
 $C^{\infty}(\T^d;\R^d)$ and converges to $\bv$ in the way claimed above
 as $\delta\searrow 0$.
\end{proof}

Then we have
\begin{lemma}\label{lem_R_conjugate}
For any $(u,\bv)\in H_{\mathrm{ave}}^{-1}(\T^d)\times
L^{p/(p-1)}(\T^d;\R^d)$
\begin{equation*}
R^*(\Phi_{p/(p-1)}((u,\bv)))=\left\{\begin{array}{ll}0 & \text{if
			      }\div\bv (\in \cD'(\R^d))\text{ satisfies
			       }\div\bv = (-\Dp)^{-1}u,\\
\infty &\text{otherwise. }
\end{array}
\right.
\end{equation*}
\end{lemma}
\begin{proof}
Take any $(u,\bv)\in H_{\mathrm{ave}}^{-1}(\T^d)\times
L^{p/(p-1)}(\T^d;\R^d)$.
\begin{equation}\label{eq_R_conjugate_pre}
\begin{split}
&R^*(\Phi_{p/(p-1)}((u,\bv)))\\
&=\sup_{(f,\bg)\in H_{\mathrm{ave}}^{-1}(\T^d)\times
L^{p}(\T^d;\R^d)}\{\<\Phi_{p/(p-1)}((u,\bv)),(f,\bg)\>-R((f,\bg))\}\\
&=\sup_{f\in
 X_{\mathrm{per}}}\left\{\<u,f\>_{H_{\mathrm{ave}}^{-1}(\T^d)}+\int_{\Op}\<\bv(\bx),\nabla \wf(\bx)\>_{\R^d}d\bx\right\}\\
&\ge \sup_{\phi\in
 C^{\infty}_{\mathrm{ave}}(\T^d)}\left\{\int_{\Op}(-\Dp)^{-1}u(\bx)\cdot\phi(\bx)d\bx+\int_{\Op}\<\bv(\bx),\nabla \phi(\bx)\>_{\R^d}d\bx\right\}\\
&=\sup_{\by\in\R^d}\sup_{\phi\in
 C^{\infty}_{\mathrm{ave}}(\T^d)}\int_{\Op+\by}((-\Dp)^{-1}u(\bx)\cdot\phi(\bx)+\<\bv(\bx),\nabla\phi(\bx)\>_{\R^d})d\bx\\
&\ge \sup_{\by\in\R^d}\sup_{\phi\in
 C^{\infty}_{0}(\Op+\by)}\int_{\Op+\by}((-\Dp)^{-1}u(\bx)\cdot\phi(\bx)+\<\bv(\bx),\nabla\phi(\bx)\>_{\R^d})d\bx\\
&=\left\{\begin{array}{ll}0 & \text{if }\div \bv(\in\cD'(\Op+\by))\text{
   satisfies }\div\bv =(-\Dp)^{-1}u|_{\Op+\by}\ (\forall \by\in\R^d),\\
\infty &\text{otherwise,}  
\end{array} \right.
\end{split}
\end{equation}
where we have used the fact that
 $\int_{\Op+\by}(-\Dp)^{-1}u(\bx)d\bx=0$.
From the inequality \eqref{eq_R_conjugate_pre} we can deduce that
\begin{equation}\label{eq_R_conjugate_one}
R^*(\Phi_{p/(p-1)}((u,\bv)))\ge \left\{\begin{array}{ll}0&\text{if }\div
				 \bv(\in\cD'(\R^d))\text{ satisfies
				  }\div \bv=(-\Dp)^{-1}u,\\
\infty&\text{otherwise.}
\end{array}
\right.
\end{equation}
To confirm this, assume that $\div \bv(\in\cD'(\Op+\by))$ satisfies
$\div\bv =(-\Dp)^{-1}u|_{\Op+\by}$ $(\forall \by\in\R^d)$. For any
 proposition P let $1_{\mathrm{P}}$ $(\in\{0,1\})$ be defined by
$$
1_{\mathrm{P}}:=\left\{\begin{array}{ll} 1 & \text{if P is true,}\\
                          0 & \text{otherwise.}
\end{array}\right.
$$
Take a function $\eta\in C^{\infty}_0(\R)$ such that $0\le \eta(x)\le
 1$ $(\forall x\in \R)$, $\eta(x)=0$ if $|x|\ge 1$ and
 $\int_{\R}\eta(x)dx=1$. By using $\eta$ we define functions $f_{i,n}\in
 C_0^{\infty}(\R)$ $(i\in\{1,\cdots,d\},n\in\Z)$ by
$$
f_{i,n}(x):=\int_{\R}\left(\frac{\o_i}{8}\right)^{-1}\eta\left(\frac{x-y}{\o_i/8}\right)1_{y\in
 [0,\o_i/2)+\o_in/2}dy.
$$
Remark that $\supp f_{i,n}\subset (0,\o_i)+\o_in/2-\o_i/4$ and
 $\sum_{n\in\Z}f_{i,n}(x)=1$ $(\forall x\in\R)$. For any $\bn$
 $(=(n_1,\cdots,n_d))\in\Z^d$ set
 $f_{\bn}(\bx):=\prod_{i=1}^df_{i,n_i}(x_i)$. We see that
 $f_{\bn}\in C_0^{\infty}(\Op+\by_{\bn})$ and
 $\sum_{\bn\in\Z^d}f_{\bn}(\bx)=1$ $(\forall \bx\in \R^d)$, where
 $\by_{\bn}:=(\o_1n_1/2-\o_1/4,\o_2n_2/2-\o_2/4,\cdots,\o_dn_d/2-\o_d/4)$ $(\in\R^d)$.
For any $\phi\in C_0^{\infty}(\R^d)$ there exists $N\in\N$ such that
$$
\supp \phi\subset \bigcup_{\bn\in\Z^d\atop |n_i|\le N\ (i=1,\cdots,d)}(\Op+\by_{\bn}).
$$
Then for any $\bx\in\supp\phi$
$$
\sum_{\bn\in\Z^d\atop |n_i|\le N+1\ (i=1,\cdots,d)}f_{\bn}(\bx)=1.
$$
Thus, by assumption
\begin{equation*}
\begin{split}
\int_{\R^d}\<\bv(\bx),-\nabla
 \phi(\bx)\>_{\R^d}d\bx&=\sum_{\bn\in\Z^d\atop |n_i|\le N+1\
 (i=1,\cdots,d)}\int_{\Op+\by_{\bn}}\<\bv(\bx),-\nabla(f_{\bn}(\bx)\phi(\bx))\>_{\R^d}d\bx\\
&=\sum_{\bn\in\Z^d\atop |n_i|\le N+1\
 (i=1,\cdots,d)}\int_{\Op+\by_{\bn}}(-\Dp)^{-1}u(\bx)\cdot f_{\bn}(\bx)\phi(\bx)d\bx\\
&=\int_{\R^d}(-\Dp)^{-1}u(\bx)\cdot \phi(\bx)d\bx.
\end{split}
\end{equation*}
Hence, $\div \bv$ $(\in\cD'(\R^d))$ satisfies $\div\bv
 =(-\Dp)^{-1}u$, which means that the right hand side of
 \eqref{eq_R_conjugate_pre} is larger than equal to that of
 \eqref{eq_R_conjugate_one}, resulting in the inequality
 \eqref{eq_R_conjugate_one}.

To show that the inequality \eqref{eq_R_conjugate_one} is actually the
 equality, let us assume that $\div\bv=(-\Dp)^{-1}u$. By Lemma
 \ref{lem_density_div} we can take a sequence $\{\bv_n\}_{n=1}^{\infty}$
 $(\subset C^{\infty}(\T^d;\R^d))$ such that $\bv_n\to \bv$ in
 $L^{p/(p-1)}(\T^d;\R^d)$, $\div \bv_n\to \div \bv$ in
 $H_{\mathrm{ave}}^{1}(\T^d)$ as $n\to \infty$. Applying Lemma
 \ref{lem_integration_by_parts}, we observe that 
\begin{equation*}
\begin{split}
R^*(\Phi_{p/(p-1)}((u,\bv)))&=\sup_{f\in X_{\mathrm{per}}}\left\{\<\div
 \bv,f\>+\int_{\Op}\<\bv(\bx),\nabla\wf(\bx)\>_{\R^d}d\bx\right\}\\
&=\sup_{f\in X_{\mathrm{per}}}\lim_{n\to \infty}\left\{\int_{\Op}(\div
 \bv_n(\bx)\wf(\bx)+\<\bv_n(\bx),\nabla\wf(\bx)\>_{\R^d})d\bx\right\}\\
&=0,
\end{split}
\end{equation*}
which concludes the proof.
\end{proof}

For any $u\in H_{\mathrm{ave}}^{-1}(\T^d)$ let $Y_{\mathrm{per}}(u)$
$(\subset  L^{p/(p-1)}(\T^d;\R^d))$ be defined by
$$
Y_{\mathrm{per}}(u):=\{\bs\in
 L^{p/(p-1)}(\T^d;\R^d)\ |\  \div \bs=(-\Dp)^{-1}u\}.
$$

Using Lemma \ref{lem_Q_conjugate} and Lemma \ref{lem_R_conjugate}, we
show the following.
\begin{lemma}\label{lem_Q_R_conjugate}
For any $(u,\bv)\in H_{\mathrm{ave}}^{-1}(\T^d)\times
 L^{p/(p-1)}(\T^d;\R^d)$
$$
(Q+R)^*(\Phi_{p/(p-1)}((u,\bv)))=\left\{\begin{array}{ll}
\displaystyle\min_{\bs\in
 Y_{\mathrm{per}}(u)}\int_{\Op}\s^{\#}(\bv(\bx)-\bs(\bx))d\bx&\text{if
  }Y_{\mathrm{per}}(u)\neq\emptyset,\\
\infty&\text{otherwise.}
\end{array}
\right.
$$
\end{lemma}

\begin{remark}
A direct application of the general theorem \cite[\mbox{Proposition
 3.4}]{A} on inf-convolution can shorten the proof of Lemma
 \ref{lem_Q_R_conjugate} below. However, we prove the lemma by referring
 only to the basic facts Lemma \ref{lem_standard_properties_Banach} and Lemma
 \ref{lem_standard_properties_Hilbert} for self-containedness of the paper.
\end{remark}

\begin{proof}
Define a functional $S:H_{\mathrm{ave}}^{-1}(\T^d)\times
 L^{p/(p-1)}(\T^d;\R^d)\to \R\cup\{\infty\}$ by
\begin{equation}\label{eq_S_def}
\begin{split}
&S((u,\bv))\\
&:=\inf_{(r,\bs)\in H_{\mathrm{ave}}^{-1}(\T^d)\times
 L^{p/(p-1)}(\T^d;\R^d)}\{Q^*(\Phi_{p/(p-1)}((u,\bv)-(r,\bs)))+R^*(\Phi_{p/(p-1)}((r,\bs)))\},\\
&\qquad \forall (u,\bv)\in H_{\mathrm{ave}}^{-1}(\T^d)\times
 L^{p/(p-1)}(\T^d;\R^d).
\end{split}
\end{equation}
Lemma \ref{lem_Q_conjugate} and Lemma \ref{lem_R_conjugate} imply that
\begin{equation}\label{eq_S_next}
S((u,\bv))=\left\{\begin{array}{ll}
\displaystyle\inf_{\bs\in
 Y_{\mathrm{per}}(u)}\int_{\Op}\s^{\#}(\bv(\bx)-\bs(\bx))d\bx&\text{if
  }Y_{\mathrm{per}}(u)\neq\emptyset,\\
\infty&\text{otherwise.}
\end{array}
\right.
\end{equation}
We need to show that
 $S((u,\bv))=(Q+R)^*(\Phi_{p/(p-1)}((u,\bv)))$.

By using the convexity of $Q^*$ and $R^*$ in \eqref{eq_S_def} we can
 prove that $S$ is convex as well. Moreover, from \eqref{eq_S_next} and
 \eqref{eq_sigma_conjugate} we see that $S$ is not identically
 $\infty$. To show the lower semi-continuity of $S$, let us assume that
 $(u_n,\bv_n)$ converges to $(u,\bv)$ in $H_{\mathrm{ave}}^{-1}(\T^d)\times
 L^{p/(p-1)}(\T^d;\R^d)$ as $n\to \infty$ and there is $\lambda\ge 0$
 such that $S((u_n,\bv_n))\le \lambda$ $(\forall n\in \N)$. The equality
 \eqref{eq_S_next} ensures that there exists
 $\{\bs_{i}^n\}_{i=1}^{\infty}$ $(\subset L^{p/(p-1)}(\T^d;\R^d))$
 such that $\div \bs_{i}^n=(-\Dp)^{-1}u_n$ $(\forall i\in\N)$ and 
$$
\lim_{i\to\infty}\int_{\Op}\s^{\#}(\bv_n(\bx)-\bs_i^n(\bx))d\bx=\inf_{\bs\in
 Y_{\mathrm{per}}(u_n)}\int_{\Op}\s^{\#}(\bv_n(\bx)-\bs(\bx))d\bx.
$$
There exists $\lambda'\ge 0$ such that 
$$
\int_{\Op}\s^{\#}(\bv_n(\bx)-\bs_i^n(\bx))d\bx\le \lambda+\lambda',\
 \forall i\in\N.
$$
By this inequality and \eqref{eq_sigma_conjugate} $\{\bs_i^n\}_{i=1}^{\infty}$ is bounded in
 $L^{p/(p-1)}(\T^d;\R^d)$. Thus, we can extract a subsequence
 $\{\bs_{i(l)}^n\}_{l=1}^{\infty}$ from $\{\bs_i^n\}_{i=1}^{\infty}$ so
 that $\bs_{i(l)}^n$ weakly converges to some $\bs_n$
 in $L^{p/(p-1)}(\T^d;\R^d)$ as $l\to \infty$. Moreover, Mazur's theorem
 for the space $L^{p/(p-1)}(\T^d;\R^d)$ $\times\R$ guarantees that for any
 $k\in\N$ there exist $l_k\in\N$ and $\beta_j^k\in[0,1]$
 $(j=1,\cdots,l_k)$ satisfying $\sum_{j=1}^{l_k}\beta_j^k=1$ such that as $k\to \infty$
\begin{equation*}
\begin{split}
&\sum_{j=1}^{l_k}\beta_{j}^k\bs_{i(j)}^n\to \bs_n\text{ in
 }L^{p/(p-1)}(\T^d;\R^d),\\
&\sum_{j=1}^{l_k}\beta_{j}^k\int_{\Op}\s^{\#}(\bv_n(\bx)-\bs_{i(j)}^n(\bx))d\bx\to
 \inf_{\bs\in Y_{\mathrm{per}}(u_n) }\int_{\Op}\s^{\#}(\bv_n(\bx)-\bs(\bx))d\bx.\end{split}
\end{equation*}
Furthermore, by extracting a subsequence from $\{\sum_{j=1}^{l_k}\beta_{j}^k\bs_{i(j)}^n\}_{k=1}^{\infty}$ we may assume that
 as $k\to\infty$
$$
\sum_{j=1}^{l_k}\beta_{j}^k\bs_{i(j)}^n(\bx)\to \bs_n(\bx)\text{ a.e.
 }\bx\in\R^d,
$$
where we used the same notation for simplicity. Then, by Fatou's lemma
 and the convexity of $\s^{\#}$ we have that 
\begin{equation}\label{eq_s_n_smaller}
\begin{split}
\int_{\Op}\s^{\#}(\bv_n(\bx)-\bs_n(\bx))d\bx&\le \liminf_{k\to\infty}\sum_{j=1}^{l_k}\beta_{j}^k\int_{\Op}\s^{\#}(\bv_n(\bx)-\bs_{i(j)}^n(\bx))d\bx\\
&=\inf_{\bs\in Y_{\mathrm{per}}(u_n)}\int_{\Op}\s^{\#}(\bv_n(\bx)-\bs(\bx))d\bx.\end{split}
\end{equation}
Since the set $Y_{\mathrm{per}}(w)$ is a
 convex, closed subset of $L^{p/(p-1)}(\T^d;\R^d)$ for any $w\in
 H_{\mathrm{ave}}^{-1}(\T^d)$ with $Y_{\mathrm{per}}(w)\neq \emptyset$, we obtain
\begin{equation}\label{eq_s_n_div}
\div \bs_n=(-\Dp)^{-1}u_n.
\end{equation}
By \eqref{eq_s_n_smaller}, \eqref{eq_s_n_div} we have that $\bs_n\in
 Y_{\mathrm{per}}(u_n)$ and 
\begin{equation}\label{eq_v_s_lambda}
\int_{\Op}\s^{\#}(\bv_n(\bx)-\bs_n(\bx))d\bx=\min_{\bs\in Y_{\mathrm{per}}(u_n)
}\int_{\Op}\s^{\#}(\bv_n(\bx)-\bs(\bx))d\bx\le\lambda.
\end{equation}

It follows from \eqref{eq_v_s_lambda} that $\{\bs_n\}_{n=1}^{\infty}$ is bounded in
 $L^{p/(p-1)}(\T^d;\R^d)$. By using Mazur's theorem for the space
 $H_{\mathrm{ave}}^{-1}(\T^d)\times L^{p/(p-1)}(\T^d;\R^d)\times
 L^{p/(p-1)}(\T^d;\R^d)$ one can show that there are sequences
 $\{n(j)\}_{j=1}^{\infty}$, $\{m_k\}_{k=1}^{\infty}$ $(\subset \N)$,
 $\g_j^k\in[0,1]$ $(j=1,\cdots, m_k)$ with $\sum_{j=1}^{m_k}\g_j^k=1$
 $(\forall k\in\N)$ and $\bs\in L^{p/(p-1)}(\T^d;\R^d)$ such that as $k\to\infty$
\begin{equation*}
\begin{split}
&\sum_{j=1}^{m_k}\g_j^k(u_{n(j)},\bv_{n(j)})\to (u,\bv)\text{ in
 }H_{\mathrm{ave}}^{-1}(\T^d)\times L^{p/(p-1)}(\T^d;\R^d),\\
&\sum_{j=1}^{m_k}\g_{j}^k\bs_{n(j)}\to \bs\text{ in
 }L^{p/(p-1)}(\T^d;\R^d).
\end{split}
\end{equation*}
Moreover, by taking a subsequence if necessary we may claim that as
 $k\to \infty$
$$
\sum_{j=1}^{m_k}\g_j^k\bv_{n(j)}(\bx)\to \bv(\bx),\
 \sum_{j=1}^{m_k}\g_{j}^k\bs_{n(j)}(\bx)\to \bs(\bx)\text{ a.e. }\bx\in\R^d.
$$
Then, Fatou's lemma, the convexity of $\s^{\#}$ and \eqref{eq_v_s_lambda} prove that
\begin{equation}\label{eq_s_smaller}
\int_{\Op}\s^{\#}(\bv(\bx)-\bs(\bx))d\bx\le \liminf_{k\to
 \infty}\sum_{j=1}^{m_k}\g_j^k\int_{\Op}\s^{\#}(\bv_{n(j)}(\bx)-\bs_{n(j)}(\bx))d\bx\le
 \lambda.
\end{equation}
Note that for any $\phi\in C_0^{\infty}(\R^d)$
\begin{equation*}
\begin{split}
\int_{\R^d}\<\bs(\bx),-\nabla \phi(\bx)\>_{\R^d}d\bx&=\lim_{k\to
 \infty}\sum_{j=1}^{m_k}\g_{j}^k\int_{\R^d}\<\bs_{n(j)}(\bx),-\nabla
 \phi(\bx)\>_{\R^d}d\bx\\
&=\lim_{k\to\infty}\sum_{j=1}^{m_k}\g_{j}^k\int_{\R^d}(-\Dp)^{-1}u_{n(j)}(\bx)\cdot
 \phi(\bx)d\bx\\
&=\int_{\R^d}(-\Dp)^{-1}u(\bx)\cdot \phi(\bx)d\bx,
\end{split}
\end{equation*}
which means that
\begin{equation}\label{eq_s_div}
\div \bs = (-\Dp)^{-1}u.
\end{equation}
By combining \eqref{eq_s_smaller}, \eqref{eq_s_div} with
 \eqref{eq_S_next} we arrive at $S((u,\bv))\le \lambda$, which concludes
 that $S$ is lower semi-continuous.

Since $S:H_{\mathrm{ave}}^{-1}(\T^d)\times
 L^{p/(p-1)}(\T^d;\R^d)\to\R\cup\{\infty\}$ is convex, lower
 semi-continuous and not identically $\infty$, we can apply Lemma
 \ref{lem_standard_properties_Banach} (\ref{double_conjugate}) to deduce
 that
\begin{equation}\label{eq_S_double_conjugate}
(S^*)^*((u,\bv))=S((u,\bv)),\ \forall (u,\bv)\in
 H_{\mathrm{ave}}^{-1}(\T^d)\times L^{p/(p-1)}(\T^d;\R^d).
\end{equation}
 
In order to characterize $S^*$ $(:(H_{\mathrm{ave}}^{-1}(\T^d)\times
 L^{p/(p-1)}(\T^d;\R^d))^*\to\R\cup\{\infty\})$, take any $(f,\bg)\in
 H_{\mathrm{ave}}^{-1}(\T^d)\times L^{p}(\T^d;\R^d)$. Recalling
 \eqref{eq_S_def}, we observe that
\begin{equation}\label{eq_S_conjugate_Q_R}
\begin{split}
&S^*(\Phi_p((f,\bg)))\\
&=\sup_{(u,\bv)\in H_{\mathrm{ave}}^{-1}(\T^d)\times
 L^{p/(p-1)}(\T^d;\R^d)}\{\<\Phi_p((f,\bg)),(u,\bv)\>-S((u,\bv))\}\\
&=\sup_{(u,\bv)\in H_{\mathrm{ave}}^{-1}(\T^d)\times
 L^{p/(p-1)}(\T^d;\R^d)}\sup_{(r,\bs)\in H_{\mathrm{ave}}^{-1}(\T^d)\times
 L^{p/(p-1)}(\T^d;\R^d)}\\
&\quad\cdot\{\<\Phi_p((f,\bg)),(u,\bv)\>-Q^*(\Phi_{p/(p-1)}((u,\bv)-(r,\bs)))-R^*(\Phi_{p/(p-1)}((r,\bs)))\}\\
&=\sup_{(r,\bs)\in H_{\mathrm{ave}}^{-1}(\T^d)\times
 L^{p/(p-1)}(\T^d;\R^d)}\sup_{(u,\bv)\in H_{\mathrm{ave}}^{-1}(\T^d)\times
 L^{p/(p-1)}(\T^d;\R^d)}\\
&\quad\cdot\{\<\Phi_p((f,\bg)),(u,\bv)-(r,\bs)\>-Q^*(\Phi_{p/(p-1)}((u,\bv)-(r,\bs)))\\
&\qquad+\<\Phi_p((f,\bg)),(r,\bs)\>-R^*(\Phi_{p/(p-1)}((r,\bs)))\}\\
&=(Q^*)^*((f,\bg))+(R^*)^*((f,\bg))\\
&=Q((f,\bg))+R((f,\bg)).
\end{split}
\end{equation}
To derive the last equality of \eqref{eq_S_conjugate_Q_R} we applied Lemma \ref{lem_standard_properties_Banach}
 (\ref{double_conjugate}) to $Q$, $R$.
Moreover, by using \eqref{eq_S_conjugate_Q_R} one can verify that
 for $(u,\bv)\in H_{\mathrm{ave}}^{-1}(\T^d)\times
 L^{p/(p-1)}(\T^d;\R^d)$ 
\begin{equation}\label{eq_S_double_conjugate_Q_R}
\begin{split}
(S^*)^*((u,\bv))&=\sup_{(f,\bg)\in H_{\mathrm{ave}}^{-1}(\T^d)\times
 L^{p}(\T^d;\R^d)}\{\<(u,\bv),\Phi_p((f,\bg))\>-S^*(\Phi_p((f,\bg)))\}\\
&=\sup_{(f,\bg)\in H_{\mathrm{ave}}^{-1}(\T^d)\times
 L^{p}(\T^d;\R^d)}\{\<\Phi_{p/(p-1)}((u,\bv)),(f,\bg)\>-(Q+R)((f,\bg))\}\\
&=(Q+R)^*(\Phi_{p/(p-1)}((u,\bv))).
\end{split}
\end{equation}
Combining \eqref{eq_S_double_conjugate_Q_R} with
 \eqref{eq_S_double_conjugate} yields
$$
S((u,\bv))=(Q+R)^*(\Phi_{p/(p-1)}(u,\bv)),\ \forall (u,\bv)\in H_{\mathrm{ave}}^{-1}(\T^d)\times
 L^{p/(p-1)}(\T^d;\R^d).
$$

Finally remark that the argument leading to \eqref{eq_v_s_lambda}
 essentially showed that ` $\inf$ ' in \eqref{eq_S_next} can be replaced
 by ` $\min$ ', which results in the desired equality.
\end{proof}

Lemma \ref{lem_F_per_conjugate_pre} and Lemma \ref{lem_Q_R_conjugate}
complete the characterization of $F_{\mathrm{per}}^{\#}$.
\begin{lemma}\label{lem_F_per_conjugate}
For any $u\in H_{\mathrm{ave}}^{-1}(\T^d)$
$$
F_{\mathrm{per}}^{\#}(u)=\left\{\begin{array}{ll}
\displaystyle\min_{\bs\in
 Y_{\mathrm{per}}(-u)}\int_{\Op}\s^{\#}(\bs(\bx))d\bx&\text{if
  }Y_{\mathrm{per}}(-u)\neq\emptyset,\\
\infty&\text{otherwise.}
\end{array}
\right.
$$
\end{lemma}

All the preparations have been done to prove Theorem
\ref{thm_subdifferential_per}.
\begin{proof}[Proof of Theorem \ref{thm_subdifferential_per}]
Assume that $\partial F_{\mathrm{per}}(f)\neq \emptyset$ throughout the
 proof. If $u\in \partial F_{\mathrm{per}}(f)$, according to Lemma
 \ref{lem_standard_properties_Hilbert} we equivalently have that
$$
F_{\mathrm{per}}(f)+F_{\mathrm{per}}^{\#}(u)=\<f,u\>_{H_{\mathrm{ave}}^{-1}(\T^d)},
$$
or by Lemma \ref{lem_F_per_conjugate} that
\begin{equation*}
F_{\mathrm{per}}(f)+\min_{\bs\in Y_{\mathrm{per}}(-u)}\int_{\Op}\s^{\#}(\bs(\bx))d\bx=\<f,u\>_{H_{\mathrm{ave}}^{-1}(\T^d)}.
\end{equation*}
Let $\bg\in Y_{\mathrm{per}}(-u)$ be a minimizer. We have $-\div \bg=(-\Dp)^{-1}u$ and 
$$
F_{\mathrm{per}}(f)+\int_{\Op}\s^{\#}(\bg(\bx))d\bx=\<f,u\>_{H_{\mathrm{ave}}^{-1}(\T^d)},
$$
which lead to
\begin{equation}\label{eq_proof_per_integral}
\int_{\Op}(\s(\nabla\wf(\bx))+\s^{\#}(\bg(\bx)))d\bx=\<f,-\div \bg\>.
\end{equation}
 
By Lemma \ref{lem_density_div} we can choose a sequence
 $\{\bg_n\}_{n=1}^{\infty}$ $(\subset C^{\infty}(\T^d;\R^d))$ so that
as $n\to\infty$
\begin{equation*}
\begin{split}
&\bg_n\to\bg\text{ in }L^{p/(p-1)}(\T^d;\R^d),\\
&\div \bg_n\to \div \bg\text{ in }H_{\mathrm{ave}}^{1}(\T^d).
\end{split}
\end{equation*}
Then by using Lemma 2.6 we see that
\begin{equation*}
\begin{split}
\<f,-\div \bg\>&=-\lim_{n\to\infty}\int_{\Op}\wf(\bx)\div
 \bg_n(\bx)d\bx=\lim_{n\to\infty}\int_{\Op}\<\nabla\wf(\bx),\bg_n(\bx)\>_{\R^d}d\bx\\
&=\int_{\Op}\<\nabla
 \wf(\bx),\bg(\bx)\>_{\R^d}d\bx.
\end{split}
\end{equation*}
Therefore, we can deduce from \eqref{eq_proof_per_integral} that
$$
\int_{\Op}(\s(\nabla \wf(\bx))+\s^{\#}(\bg(\bx))-\<\nabla
 \wf(\bx),\bg(\bx)\>_{\R^d})d\bx= 0.
$$
Since the integrand of the integral above is non-negative, we obtain
$$
\s(\nabla
 \wf(\bx))+\s^{\#}(\bg(\bx))-\<\nabla\wf(\bx),\bg(\bx)\>_{\R^d}=0\text{
  a.e. }\bx\in\R^d,
$$
or equivalently 
 $\bg(\bx)\in\partial\s(\nabla \wf(\bx))$ a.e. $\bx\in\R^d$ by Lemma \ref{lem_standard_properties_Hilbert}. Since
 $u=-(-\Dp)$ $\cdot\div\bg$, we have
 proved the inclusion ` $\subset$ ' of the claim of Theorem
 \ref{thm_subdifferential_per}.

To show the opposite inclusion ` $\supset$ ', take any $u\in
 H_{\mathrm{ave}}^{-1}(\T^d)$ for which there is $\bg\in
 L^{p/(p-1)}(\T^d;\R^d)$ such that $\div \bg\in
 H_{\mathrm{ave}}^{1}(\T^d)$, $u=-(-\Dp)\div \bg$ and
 $\bg(\bx)\in\partial\s(\nabla \wf(\bx))$ a.e. $\bx\in\R^d$. Then by
 exactly following the argument above the other way round we can reach
$$
F_{\mathrm{per}}(f)+\int_{\Op}\s^{\#}(\bg(\bx))d\bx=\<f,u\>_{H_{\mathrm{ave}}^{-1}(\T^d)}.
$$
By taking infimum over such $\bg$s and by Lemma
 \ref{lem_F_per_conjugate} one has
$$F_{\mathrm{per}}(f)+F^{\#}_{\mathrm{per}}(u)\le \<f,u\>_{H_{\mathrm{ave}}^{-1}(\T^d)},$$
which is equivalent to the inclusion $u\in\partial F_{\mathrm{per}}(f)$
 by the definition of $F_{\mathrm{per}}^{\#}$ and Lemma
 \ref{lem_standard_properties_Hilbert}. We have proved the
 inclusion ` $\supset$ ' as well.
\end{proof}

\subsection{Proof for the Dirichlet problem}
The major part of the proof for Theorem \ref{thm_subdifferential_D} can
be constructed by straightforwardly translating the proof for Theorem
\ref{thm_subdifferential_per} into the context with the Dirichlet
boundary condition. Let us, therefore, explain only different parts
from the periodic problem and be brief about the parallel parts.

To characterize the conjugate functional $F_{\mathrm{D}}^{\#}$ $(:H^{-1}(\O)\to\R\cup\{\infty\})$ we
introduce functionals $\Qd,\Rd:H^{-1}(\O)\times
L^p(\O;\R^d)\to\R\cup\{\infty\}$ by
\begin{equation*}
\begin{split}
&\Qd((f,\bg)):=\int_{\O}\s(\bg(\bx))d\bx,\\
&\Rd((f,\bg)):=\left\{\begin{array}{ll}0&\text{if }f\in
		X_{\mathrm{D}}\text{ and }\bg=\nabla\wf,\\
\infty&\text{otherwise,}
\end{array}
\right. 
\end{split}
\end{equation*}
where $H^{-1}(\O)\times L^p(\O;\R^d)$ is the real Banach space with the
norm $\|(f,\bg)\|_{\mathrm{D}}:=\|f\|_{H^{-1}(\O)}+\|\bg\|_{L^p(\O;\R^d)}$. The functionals $\Qd$, $\Rd$ are convex, lower
semi-continuous and not identically $\infty$. 

The difference from the
periodic problem mainly lies in a lack of a density property like Lemma
\ref{lem_density_div}, which worked conveniently in the periodic case.  
Consequently in the Dirichlet problem the characterization of $R_{\mathrm{D}}^*$,
$F_{\mathrm{D}}^{\#}$ and $\partial F_{\mathrm{D}}$ inherits an
additional constraint, which is to require a function $\bw$ $(\in
L^{p/(p-1)}(\O;\R^d))$ satisfying $\div \bw\in H_0^1(\O)$ to obey
\begin{equation}\label{eq_additional}
\<\div\bw,h\>+\int_{\O}\<\bw(\bx),\nabla\widetilde{h}(\bx)\>_{\R^d}d\bx=0,\
 \forall h\in X_{\mathrm{D}}.
\end{equation}

The first difference appears in the characterization of
$R_{\mathrm{D}}^*$ $(:(H^{-1}(\O)\times L^p(\O;\R^d))^*$ $\to
\R\cup\{\infty\})$, while the characterization of $Q_{\mathrm{D}}^*$ can
be carried out in the same way as in Lemma \ref{lem_Q_conjugate}. Using
the isomorphism $\Psi_{p/(p-1)}:H^{-1}(\O)\times L^{p/(p-1)}(\O;\R^d)\to
(H^{-1}(\O)\times L^p(\O;\R^d))^*$ defined by
\begin{equation*}
\begin{split}
&\<\Psi_{p/(p-1)}((u,\bv)),(f,\bg)\>:=\<u,f\>_{H^{-1}(\O)}+\int_{\O}\<\bv(\bx),\bg(\bx)\>_{\R^d}d\bx,\\
&\forall (u,\bv)\in H^{-1}(\O)\times L^{p/(p-1)}(\O;\R^d),\ \forall
 (f,\bg)\in H^{-1}(\O)\times L^{p}(\O;\R^d),
\end{split}
\end{equation*}
we have
\begin{lemma}\label{lem_R_D_conjugate}
For any $(u,\bv)\in H^{-1}(\O)\times L^{p/(p-1)}(\O;\R^d)$
$$
R_{\mathrm{D}}^*(\Psi_{p/(p-1)}((u,\bv)))=\left\{\begin{array}{ll}0&
\text{if }\bv\text{ satisfies \eqref{eq_additional} and }\div\bv
					  =(-\Dd)^{-1}u,\\
\infty &\text{otherwise.}
\end{array}
\right.
$$ 
\end{lemma}
\begin{proof}
Take any $(u,\bv)\in H^{-1}(\O)\times L^{p/(p-1)}(\O;\R^d)$.
\begin{equation*}
\begin{split}
R_{\mathrm{D}}^*(\Psi_{p/(p-1)}((u,\bv)))&=\sup_{f\in
 X_{\mathrm{D}}}\left\{\<u,f\>_{H^{-1}(\O)}+\int_{\O}\<\bv(\bx),\nabla\wf(\bx)\>_{\R^d}d\bx\right\}\\
&\ge \sup_{\phi\in C_{0}^{\infty}(\O)}\int_{\O}((-\Dd)^{-1}u(\bx)\cdot
 \phi(\bx)+\<\bv(\bx),\nabla \phi(\bx)\>_{\R^d})d\bx\\
&=\left\{\begin{array}{ll}0&\text{if }\div\bv=(-\Dd)^{-1}u,\\
\infty&\text{otherwise. }
\end{array}\right.
\end{split}
\end{equation*}

On the assumption that $\div \bv=(-\Dd)^{-1}u$ we have that
\begin{equation*}
\begin{split}
R_{\mathrm{D}}^*(\Psi_{p/(p-1)}((u,\bv)))&=\sup_{f\in
 X_{\mathrm{D}}}\left\{\<\div \bv,f\>+\int_{\O}\<\bv(\bx),\nabla\wf(\bx)\>_{\R^d}d\bx\right\}\\
&=\left\{\begin{array}{ll}0&\text{if }\bv\text{ satisfies \eqref{eq_additional}},\\
\infty&\text{otherwise.}
\end{array}
\right.
\end{split}
\end{equation*}
\end{proof}

For any $u\in H^{-1}(\O)$ let us define a subset $Y_{\mathrm{D}}(u)$ of
$L^{p/(p-1)}(\O;\R^d)$ by
$$
Y_{\mathrm{D}}(u):=\{\bs\in L^{p/(p-1)}(\O;\R^d)\ |\ \div\bs =(-\Dd)^{-1}u,\
\bs\text{ satisfies \eqref{eq_additional}}\}.
$$
By noting that $Y_{\mathrm{D}}(u)$
is convex and closed in $L^{p/(p-1)}(\O;\R^d)$ for any $u\in H^{-1}(\O)$
with $Y_{\mathrm{D}}(u)\neq \emptyset$, we can straightforwardly
modify the proof of Lemma \ref{lem_Q_R_conjugate} to conclude the following.

\begin{lemma}\label{lem_Q_R_D_conjugate}
 For any $(u,\bv)\in H^{-1}(\O)\times L^{p/(p-1)}(\O;\R^d)$
$$
(Q_{\mathrm{D}}+R_{\mathrm{D}})^*(\Psi_{p/(p-1)}((u,\bv)))=\left\{\begin{array}{ll}\displaystyle\min_{\bs\in
							    Y_{\mathrm{D}}(u)}\int_{\O}\s^{\#}(\bv(\bx)-\bs(\bx))d\bx&\text{if }Y_{\mathrm{D}}(u)\neq\emptyset,\\
\infty&\text{otherwise.}
\end{array}
\right.
$$
\end{lemma}

Since the Dirichlet analogue of Lemma \ref{lem_F_per_conjugate_pre}
holds naturally, we obtain from Lemma \ref{lem_Q_R_D_conjugate} that 

\begin{lemma}\label{lem_F_D_conjugate}
For any $u\in H^{-1}(\O)$
$$
F_{\mathrm{D}}^{\#}(u)=\left\{\begin{array}{ll}\displaystyle\min_{\bs\in
							    Y_{\mathrm{D}}(-u)}\int_{\O}\s^{\#}(\bs(\bx))d\bx&\text{if }Y_{\mathrm{D}}(-u)\neq\emptyset,\\
\infty&\text{otherwise.}
\end{array}
\right.
$$
\end{lemma}
On these preparations we can prove Theorem \ref{thm_subdifferential_D}.
\begin{proof}[Proof of Theorem \ref{thm_subdifferential_D}] 
Assume that $\partial F_{\mathrm{D}}(f)\neq\emptyset$. By Lemma
 \ref{lem_standard_properties_Hilbert} and Lemma
 \ref{lem_F_D_conjugate} the inclusion
 $u\in\partial F_{\mathrm{D}}(f)$ is equivalent to the equality
\begin{equation}\label{eq_proof_D_min}
F_{\mathrm{D}}(f)+\min_{\bs\in Y_{\mathrm{D}}(-u)}\int_{\O}\s^{\#}(\bs(\bx))d\bx=\<f,u\>_{H^{-1}(\O)}.
\end{equation}
If $\bg$ $(\in Y_{\mathrm{D}}(-u))$ is a minimizer,
$u=-(-\Dd)\div\bg$ and the equality \eqref{eq_proof_D_min} coupled
 with \eqref{eq_additional} leads to 
$$
\int_{\O}\s(\nabla
 \wf(\bx))d\bx+\int_{\O}\s^{\#}(\bg(\bx))d\bx=\int_{\O}\<\bg(\bx),\nabla \wf(\bx)\>_{\R^d}d\bx,
$$
which is equivalent to the inclusion that $\bg(\bx)\in\partial
 \s(\nabla\wf(\bx))$ a.e. $\bx\in\O$ by Lemma
 \ref{lem_standard_properties_Hilbert}. We have proved the inclusion `
 $\subset$ ' of Theorem \ref{thm_subdifferential_D}. The opposite
 inclusion ` $\supset$ ' can be shown by arguing the other way around.
\end{proof}

\begin{proof}[Proof of Corollary \ref{cor_subdifferential_D_restricted}]
We show that 
\begin{equation}\label{eq_integral_functional}
\begin{split}
&\bullet \text{the bilinear form }(f,g)\mapsto
      \int_{\O}f(\bx)g(\bx)d\bx\text{ is
      well-defined on }W_0^{1,p}(\O)\times H_0^1(\O),\\
&\bullet f\mapsto \int_{\O}f(\bx)g(\bx)d\bx\text{ is continuous in }
      W_0^{1,p}(\O)\ (\forall g\in H_0^1(\O)),\\
&\bullet g\mapsto \int_{\O}f(\bx)g(\bx)d\bx\text{ is continuous in }
      H_0^1(\O)\ (\forall f\in W_0^{1,p}(\O)),
\end{split}
\end{equation}
in the assumed circumstance by means of the Sobolev embedding theorem.
Note that if \eqref{eq_integral_functional} holds, the constraint
 \eqref{eq_additional} is trivial by the density property of
 $C_0^{\infty}(\O)$ in $H_0^1(\O)$ and $W_0^{1,p}(\O)$.

If $d\le 2$, $H_0^1(\O)\subset L^{p/(p-1)}(\O)$, thus
 \eqref{eq_integral_functional} is true.

If $p\ge d$, $W_0^{1,p}(\O)\subset L^2(\O)$. Therefore
 \eqref{eq_integral_functional} holds.

If $d\ge 3$ and $1<p<d$, $H_0^1(\O)\subset L^{2d/(d-2)}(\O)$ and
 $W_0^{1,p}(\O)\subset L^{dp/(d-p)}(\O)$. From this we see that the
 inequality
$$
\frac{2d}{d-2}\ge \frac{dp/(d-p)}{dp/(d-p)-1}
$$
is sufficient to guarantee \eqref{eq_integral_functional}. This inequality
 is equivalent to $p\ge 2d/(d+4)$.

By summing up, the condition \eqref{eq_condition_D} is seen to be
 sufficient for \eqref{eq_integral_functional} to be true.
\end{proof}

\section{Canonical restriction for a spherically symmetric surface}\label{sec_canonical_restriction}
In this section we will find the smallest element in $\partial
F_{\mathrm{D}}(f)$ with respect to the norm $\|\cdot\|_{H^{-1}(\O)}$ by
giving a spherically symmetric surface $\wf$. Let us write the smallest
element called canonical restriction as $\partial
F_{\mathrm{D}}^c(f)$. It is known (see e.g. \cite{B}) that the solution
to the initial value problem
$$\left\{\begin{array}{l}
\frac{d}{dt}f_{\mathrm{D}}(t)\in -\partial
 F_{\mathrm{D}}(f_{\mathrm{D}}(t))\text{ a.e.}t>0,\\
f_{\mathrm{D}}(0)=f_{\mathrm{D},0}(\in \overline{X_{\mathrm{D}}})
\end{array}
\right.
$$
satisfies
$$
\frac{d^+}{dt}f_{\mathrm{D}}(t)=-\partial
F_{\mathrm{D}}^c(f_{\mathrm{D}}(t))\text{ all }t>0,
$$
where $d^+/dt$ means the right derivative. Hence, the canonical
restriction provides useful information on the time evolution of the
crystalline surface as already discussed for 1 dimensional problems in 
\cite[\mbox{Section 4}]{K}, \cite[\mbox{Chapter 3}]{O}. Here we argue 
a general dimensional problem under the constraint
\eqref{eq_condition_D}.

We fix $f\in X_{\mathrm{D}}$ whose $\wf$ $(\in W_0^{1,p}(\O))$ satisfies that $\nabla
\wf(\bx)=0$ a.e. $\bx\in \O_0$ and $\nabla \wf(\bx)\neq 0$ a.e. $\bx \in
\O \backslash \overline{\O_0}$ with an open set $\O_0$ satisfying 
$\overline{\O_0}\subset \O$.

Using this $\wf$, we define a function
$\bu_{\wf}:\O\backslash\overline{\O_0}\to \R^d$ by
$$
\bu_{\wf}(\bx):=|\nabla \wf(\bx)|^{-1}\nabla \wf(\bx)+\mu |\nabla \wf(\bx)|^{p-2}\nabla \wf(\bx).
$$

For any functions $\bg:\O_0\to \R^m$, $\bh:\O\backslash \overline{\O_0}\to \R^m$ $(m\in\N)$, let
$(\bg|\bh):\O\to\R^m$ be defined by
$$
(\bg|\bh)(\bx):=\left\{\begin{array}{ll}\bg(\bx)&\text{ if }\bx\in\O_0,
		 \\ \bh(\bx)&\text{ if }\bx\in\O\backslash \overline{\O_0}.
\end{array}
\right.
$$
The following lemma tells us a way to find $\partial
F_{\mathrm{D}}^c(f)$.

\begin{lemma}\label{lem_search_canonical_restriction} 
Assume that the condition \eqref{eq_condition_D} holds and that $\bg$ $(\in
 L^{p/(p-1)}(\O_0;\R^d))$ and $\bu_{\wf}$ $(\in
 L^{p/(p-1)}(\O\backslash \overline{\O_0};\R^d))$ satisfy the
 following conditions.
\begin{enumerate}[(i)]
\item\label{item_smooth}
there exists $\bpsi\in C^{\infty}_0(\O;\R^d)$ such that
     $\bpsi|_{\O_0}=\bg$.
\item\label{item_E_L}
$\nabla \Delta \div \bg(\bx)=\b0$, $\forall \bx\in \O_0$.
\item\label{item_bdd}
$|\bg(\bx)|\le 1$, $\forall \bx\in \O_0$.
\item\label{item_conforming}
$\div (\bg|\bu_{\wf})\in H_0^1(\O)$. 
\end{enumerate}
Then, $\partial F_{\mathrm{D}}^c(f)=-(-\Dd)\div (\bg|\bu_{\wf})$.
\end{lemma}
\begin{proof}
By the conditions \eqref{item_bdd}, \eqref{item_conforming} and
 Corollary \ref{cor_subdifferential_D_restricted}, $-(-\Dd)\div
 (\bg|\bu_{\wf})\in\partial F_{\mathrm{D}}(f)$. Since $\partial
 F_{\mathrm{D}}(f)$ is a non-empty, closed convex set in $H^{-1}(\O)$,
 the canonical restriction $\partial F_{\mathrm{D}}^c(f)$ uniquely
 exists. By Corollary \ref{cor_subdifferential_D_restricted} we may
 write $\partial F_{\mathrm{D}}^c(f)=-(-\Dd)\div \bG$ with some $\bG\in
 L^{p/(p-1)}(\O;\R^d)$ satisfying
 $\bG|_{\O\backslash\overline{\O_0}}=\bu_{\wf}$. By convexity of
 $\partial F_{\mathrm{D}}(f)$ and minimality of $\|-\Dd \div
 \bG\|_{H^{-1}(\O)}$ we have that
\begin{equation}\label{eq_equation_minimizer}
\begin{split}
\lim_{\eps\searrow 0}&\frac{d}{d\eps}\|(1-\eps)(-\Dd)\div \bG+\eps
 (-\Dd)\div (\bg|\bu_{\wf})\|^2_{H^{-1}(\O)}\\
&=2\int_{\O_0}\<\nabla \div \bg(\bx)-\nabla \div \bG(\bx),\nabla \div
 \bG(\bx)\>_{\R^d}d\bx\ge 0.
\end{split}
\end{equation}

On the other hand, we can derive the equality that for any $\bpsi\in
 C_0^{\infty}(\O;\R^d)$
\begin{equation*}
\int_{\O_0}\<\nabla \div \bg(\bx)-\nabla \div \bG(\bx), \bpsi(\bx)\>_{\R^d}d\bx=\int_{\O_0}\<\bg(\bx)-\bG(\bx),\nabla \div \bpsi(\bx)\>_{\R^d}d\bx.
\end{equation*}
Then by the assumptions \eqref{item_smooth}, \eqref{item_E_L}
\begin{equation}\label{eq_equation_by_assumption}
\int_{\O_0}\<\nabla \div \bg(\bx)-\nabla \div\bG(\bx),\nabla \div
 \bg(\bx)\>_{\R^d}d\bx=0.
\end{equation}

Combining \eqref{eq_equation_minimizer} with
 \eqref{eq_equation_by_assumption} gives 
$$
\int_{\O_0}|\nabla \div\bg(\bx)-\nabla \div \bG(\bx)|^2d\bx\le 0,
$$
or $-\Dd \div (\bg|\bu_{\wf})=-\Dd\div \bG$.
\end{proof}

\subsection{A spherically symmetric surface}
Let us apply Lemma \ref{lem_search_canonical_restriction} to find the
canonical restriction $\partial F_{\mathrm{D}}^c(f)$ under assumptions
that both $\O_0$ and $\O$ are spherical domains and $\wf:\O\to \R$ is
spherically symmetric. More precisely we assume that
$$\O_0=\{\bx\in\R^d\ |\ |\bx|<r_0\},\  \O=\{\bx\in\R^d\ |\ |\bx|<r\}$$
with $0<r_0<r$ and $\wf(\bx):=h(|\bx|)$ with $h\in C^1([0,r])$
satisfying 
$$h(r)=0,\ h^{(1)}(s)=0\ (\forall s\in [0,r_0])\text{ and }h^{(1)}(s)<0\
(\forall s\in
(r_0,r)).$$
Here and below let the notation $u\in C^l([a,b])$
$(l\in \N\cup \{0\}, a<b)$ mean that $u\in C^l((a,b))$ and $u^{(k)}\in
C([a,b])$ $(k\in \{0,1,\cdots,l\})$. The corresponding $f$ $(\in
H^{-1}(\O))$ to this $\wf$ $(\in W_0^{1,p}(\O))$ is characterized by
$$
\<f,\phi\>=\int_{\O}\wf(\bx)\phi(\bx)d\bx,\ \forall \phi\in H_0^1(\O).
$$ 
To organize the calculation of the canonical restriction $\partial
F_{\mathrm{D}}^c(f)$ below, we define a function $H:[r_0,r]\to\R$ by
$$H(s):=-1+\mu |h^{(1)}(s)|^{p-2}h^{(1)}(s),\ \forall s\in [r_0,r].$$
\begin{theorem}\label{thm_canonical_restriction}
Assume that $H\in C^3([r_0,r])$,
\begin{align}
&H^{(1)}(r)+\frac{d-1}{r}H(r)=0,\label{eq_assumption_zero_divergence}\\
&H^{(1)}(r_0)\in
 [-9/r_0,0].\label{eq_assumption_interval}
\end{align}
Then for all $\phi\in H_0^1(\O)$
\begin{equation}\label{eq_canonical_restriction}
\begin{split}
\<\partial
F_{\mathrm{D}}^c(f),\phi\>
&=\frac{d(d+2)}{r_0^2}\left(H^{(1)}(r_0)+\frac{1}{r_0}\right)\int_{\O_0}\phi(\bx)d\bx\\
&\quad+\int_{\O\backslash\overline{\O_0}}\Bigg(H^{(3)}(|\bx|)+\frac{2(d-1)}{|\bx|}H^{(2)}(|\bx|)+\frac{(d-1)(d-3)}{|\bx|^2}H^{(1)}(|\bx|)\\
&\qquad\qquad\qquad -\frac{(d-1)(d-3)}{|\bx|^3}H(|\bx|)\Bigg)\phi(\bx)d\bx\\
&\quad+\left(H^{(2)}(r_0)-\frac{3}{r_0}H^{(1)}(r_0)-\frac{3}{r_0^2}\right)\int_{\partial \O_0}\phi(\bx)dS,
\end{split}
\end{equation}
where $dS$ denotes the surface measure.
\end{theorem}
\begin{remark}
The surface integral over $\partial \O_0$ in
 \eqref{eq_canonical_restriction} corresponds to the appearance of delta
 functions in one dimensional case \cite[\mbox{Theorem 4.1}]{K}. The
 surface integral disappears and the canonical restriction can be
 identified with a function being constant on the facet $\O_0$
 if $H^{(2)}(r_0)-3/r_0 H^{(1)}(r_0)-3/r_0^2=0$. This remark was missed in the conclusion of \cite{K} and was properly
 taken into account in \cite[\mbox{Chapter 3}]{O} during its derivation
 of the free boundary value problem.
\end{remark}
\begin{proof}
First note that $\bu_{\wf}(\bx)=H(|\bx|)\bx/|\bx|$ $(\forall \bx\in\O\backslash
 \overline{\O_0})$ and by the assumption
 \eqref{eq_assumption_zero_divergence}
 \begin{equation}\label{eq_zero_boundary}
\div \bu_{\wf}(\bx)=0,\ \forall \bx\in \partial \O.
\end{equation}

Next let us find $\bg:\O_0\to \R^d$ satisfying \eqref{item_smooth},
 \eqref{item_E_L}, \eqref{item_conforming} of Lemma
 \ref{lem_search_canonical_restriction}. Postulate that
 $\bg(\bx)=\eta(|\bx|)\bx/|\bx|$ with a function
 $\eta:[0,r_0]\to\R$. Then we have that
\begin{equation*}
\begin{split}
\nabla \Delta \div
 \bg(\bx)=&\Big(|\bx|^4\eta^{(4)}(|\bx|)+2(d-1)|\bx|^3\eta^{(3)}(|\bx|)+(d-1)(d-5)|\bx|^2\eta^{(2)}(|\bx|)\\
&\quad -3(d-1)(d-3)|\bx|\eta^{(1)}(|\bx|)+3(d-1)(d-3)\eta(|\bx|)\Big)\frac{\bx}{|\bx|^5}.
\end{split}
\end{equation*}
The general solution to the ODE
\begin{equation*}
\begin{split}
&s^4\eta^{(4)}(s)+2(d-1)s^3\eta^{(3)}(s)+(d-1)(d-5)s^2\eta^{(2)}(s)\\
&\quad -3(d-1)(d-3)s\eta^{(1)}(s)+3(d-1)(d-3)\eta(s)=0\ (s>0)
\end{split}
\end{equation*}
is given by
$$
\eta(s)=C_1s+C_2s^3+C_3s^{-(d-1)}+C_4\left\{\begin{array}{ll}s \log s
				      &\text{if }d=2,\\
					     s^{-(d-3)}&\text{if }d\neq
					    2,\end{array}\right.\ \forall C_i\in\R\ (i=1,2,3,4).$$
Since we are looking for $\bg\in C^{\infty}(\O_0;\R^d)$, $C_3=C_4=0$. Therefore,$$
\bg(\bx)=(C_1|\bx|+C_2|\bx|^3)\frac{\bx}{|\bx|},\ \forall \bx\in\O_0.
$$
To determine $C_1,C_2$ we use the continuity conditions on
 $\partial\O_0$. Since $\div (\bg|\bu_{\wf})\in L^2(\O)$,
$\<\bg(\bx),\bx/|\bx|\>_{\R^d}=\<\bu_{\wf}(\bx), \bx/|\bx|\>_{\R^d}$
 $(\forall \bx\in \partial \O_0)$, or $\eta(r_0)=H(r_0)$. Coupling this with the fact $h^{(1)}(r_0)=0$ yields 
\begin{equation}\label{eq_div_continuity}
C_1r_0+C_2r_0^3=-1.
\end{equation}
Moreover, since $(\div \bg|\div \bu_{\wf})\in H_0^1(\O)$, $\div
 \bg(\bx)=\div \bu_{\wf}(\bx)$ $(\forall \bx\in \partial \O_0)$,
which implies that
 $\eta^{(1)}(r_0)+(d-1)\eta(r_0)/r_0=H^{(1)}(r_0)+(d-1)H(r_0)/r_0$, or
 by using the equality $\eta(r_0)=H(r_0)$,
\begin{equation}\label{eq_gradient_continuity}
C_1+3C_2r_0^2=H^{(1)}(r_0).
\end{equation}
By solving \eqref{eq_div_continuity}-\eqref{eq_gradient_continuity} we have
$$
\bg(\bx)=\left(\frac{1}{2r_0^2}\left(H^{(1)}(r_0)+\frac{1}{r_0}\right)|\bx|^3-\frac{1}{2}\left(H^{(1)}(r_0)+\frac{3}{r_0}\right)|\bx|\right)\frac{\bx}{|\bx|},
$$
which is seen to satisfy \eqref{item_smooth}, \eqref{item_E_L},
 \eqref{item_conforming} of Lemma \ref{lem_search_canonical_restriction}
 by its construction and \eqref{eq_zero_boundary}.

An elementary argument shows that this $\bg$ obeys \eqref{item_bdd}
 of Lemma \ref{lem_search_canonical_restriction} if and only if
 \eqref{eq_assumption_interval} holds. 

We have checked that all the requirements
 of Lemma \ref{lem_search_canonical_restriction} are fulfilled, and thus
 obtain $\partial F_{\mathrm{D}}^c(f)=-(-\Dd)\div(\bg|\bu_{\wf})$. Then
 by direct calculation we can deduce \eqref{eq_canonical_restriction}.
\end{proof}
\begin{example}
Assume that $r=2r_0$, $p=2$ and $\mu=1$. In this setting let us give a
 surface $\wf$ realizing all the assumptions of Theorem
 \ref{thm_canonical_restriction} plus
\begin{equation}\label{eq_assumption_no_delta}
H^{(2)}(r_0)-\frac{3}{r_0}H^{(1)}(r_0)-\frac{3}{r_0^2}=0,
\end{equation}
so not having the surface integral over $\partial\O_0$ in
 \eqref{eq_canonical_restriction}. 
Note that now the condition \eqref{eq_condition_D} holds and 
$$
H(s)=-1+h^{(1)}(s),\ \forall s\in [r_0,2r_0].
$$
We can summarize the assumptions of Theorem
 \ref{thm_canonical_restriction} and \eqref{eq_assumption_no_delta} in
 terms of $h$ as follows.
\begin{equation*}
\begin{split}
&h\in C^1([0,2r_0])\cap C^4([r_0,2r_0]),\\
&h(2r_0)=0,\\
&h^{(1)}(s)=0,\ \forall s\in [0,r_0],\\
&h^{(1)}(s)<0,\ \forall s\in (r_0,2r_0),\\
&h^{(2)}(2r_0)+\frac{d-1}{2r_0}(-1+h^{(1)}(2r_0))=0,\\
&h^{(2)}(r_0)\in [-9/r_0,0],\\
&h^{(3)}(r_0)-\frac{3}{r_0}h^{(2)}(r_0)-\frac{3}{r_0^2}=0.
\end{split}
\end{equation*}

Define $h:[0,2r_0]\to\R$ by
\begin{equation*}
\begin{split}
&h(s):=\\
&\left\{\begin{array}{ll}\displaystyle
\int_{2r_0}^{r_0}\left(-\frac{3}{5r_0^3}(t-r_0)(t-2r_0)^2+\frac{d-1}{2r_0^4}(t-r_0)^3(t-2r_0)\right)dt&
 \text{if }s\in[0,r_0],\\
\displaystyle\int_{2r_0}^{s}\left(-\frac{3}{5r_0^3}(t-r_0)(t-2r_0)^2+\frac{d-1}{2r_0^4}(t-r_0)^3(t-2r_0)\right)dt&
 \text{if }s\in(r_0,2r_0].
\end{array}
\right.
\end{split}
\end{equation*}
Then, $h$ obeys all the constraints listed above. With this $h$, define
 $\wf(\bx):=h(|\bx|)$ $(\forall \bx\in \O)$. By Theorem
 \ref{thm_canonical_restriction}, $\partial F_{\mathrm{D}}^c(f)\in
 L^{\infty}(\O)$ and 
\begin{equation*}
\begin{split}
\partial F_{\mathrm{D}}^c(f)(\bx)
&=\frac{2d(d+2)}{5r_0^3}1_{\bx\in\O_0}\\
&\quad
 +\Big(h^{(4)}(|\bx|)+\frac{2(d-1)}{|\bx|}h^{(3)}(|\bx|)+\frac{(d-1)(d-3)}{|\bx|^2}h^{(2)}(|\bx|)\\
&\qquad\quad
 -\frac{(d-1)(d-3)}{|\bx|^3}h^{(1)}(|\bx|)+\frac{(d-1)(d-3)}{|\bx|^3}\Big)1_{\bx\in \O\backslash\overline{\O_0}}.
\end{split}
\end{equation*}
\end{example}

\section*{Acknowledgments}
The author wishes to thank Professor Yoshikazu Giga for explaining the
articles \cite{GG}, \cite{GK} as well as recommending writing Section
\ref{sec_canonical_restriction} of this paper.


\begin{thebibliography}{0}
\bibitem{ACM}F. Andreu-Vaillo, V. Caselles and J. M. Maz\'on, ``Parabolic
	Quasilinear Equations Minimizing Linear Growth Functionals,''
	Progress in Mathematics {\bf 223}, Birkh\"auser-Verlag, Basel-Boston-Berlin, 2004.
\bibitem{A}H. Attouch, ``Variational Convergence for Functions and
	Operators,''  Pitman Advanced Publishing Program,
	Boston-London-Melbourne, 1984.
\bibitem{AD}H. Attouch and A. Damlamian, Application des m\'ethodes de
	convexit\'e et monotonie a l'\'etude de certaines \'equations quasi
	lin\'eaires, {\it Proc. Roy. Soc. Edinburgh} {\bf 79A} (1977), 107--129.
\bibitem{Au}T. Aubin, ``Some Nonlinear Problems in Riemannian
	Geometry,''  Springer-Verlag, Berlin-Heidelberg-New York, 1998.
\bibitem{B}H. Brezis, Monotonicity methods in Hilbert spaces and some
	applications to nonlinear partial differential equations, {\it
	in} ``Contributions to Nonlinear Functional Analysis''
	(E. H. Zarantonello, Eds.), pp. 101--156,  Academic Press, New York, 1971.
\bibitem{ET}I. Ekeland and R. Temam, ``Convex Analysis and Variational
	Problems,''  North-Holland, Amsterdam, 1976.
\bibitem{GG}M.-H. Giga and Y. Giga, Very singular diffusion equations:
	second and fourth order problems, {\it Japan
	J. Indust. Appl. Math} {\bf 27} (2010), 323--345.  
\bibitem{GK}Y. Giga and R. V. Kohn, Scale-invariant extinction time
	estimates for some singular diffusion equations,
        {\it Disc. Cont. Dyn. Sys} {\bf 30} (2011), 509--535.	
\bibitem{J}J. Jost, ``Riemannian Geometry and Geometric Analysis,''
	Springer-Verlag, Berlin-Heidelberg, 1995.
\bibitem{K}Y. Kashima, A subdifferential formulation of fourth order
	singular diffusion equations, {\it Adv. Math. Sci. Appl} {\bf
	14} (2004), 49--74.
\bibitem{KV}R. V. Kohn and H. M. Versieux, Numerical analysis of a
	steepest-descent PDE model for surface relaxation below the
	roughening temperature, {\it SIAM J. Numer. Anal} {\bf 48}
	(2010), 1781--1800.
\bibitem{O}I. V. Odisharia, ``Simulation and analysis of the relaxation of
	a crystalline surface,''  PhD thesis, Dep. of Phys. NYU, 2006.
\bibitem{S}H. Spohn, Surface dynamics below the roughening transition,
	{\it J. Phys. I France} {\bf 3} (1993), 69--81.
\end{thebibliography}
\end{document}